\documentclass{amsart}
\usepackage{amsthm}
\usepackage{amssymb}
\usepackage{amsmath}
\usepackage{mathrsfs}
\usepackage{hyperref}
\usepackage{chngcntr}
\usepackage{tikz}
\usetikzlibrary{matrix,arrows}

\counterwithin{equation}{section}

\makeatletter
\newtheorem*{rep@theorem}{\rep@title}
\newcommand{\newreptheorem}[2]{%
\newenvironment{rep#1}[1]{%
 \def\rep@title{#2 \ref{##1}}%
 \begin{rep@theorem}}%
 {\end{rep@theorem}}}
\makeatother

\newtheorem{thm}{Theorem}[section]
\newreptheorem{thm}{Theorem}
\newtheorem{prop}[thm]{Proposition} 

\newtheorem{lem}[thm]{Lemma}
\newtheorem{cor}[thm]{Corollary}

\theoremstyle{definition}
\newtheorem{dfn}[thm]{Definition}
\newtheorem{exmpl}[thm]{Example}
\newtheorem{?}[thm]{Question}
\newtheorem*{blankdfn}{Definition}

\theoremstyle{remark}
\newtheorem{rmk}[thm]{Remark}

\newcommand{\HOM}{\mathbb{H}\text{om}}
\newcommand{\FR}{\mathfrak}
\newcommand{\ds}{\displaystyle}
\newcommand{\tql}{\textquotedblleft}
\newcommand{\tqr}{\textquotedblright}

\begin{document}

\author{Scott Atkinson}
\title{Convex Sets Associated to $C^*$-Algebras}
\address{University of Virginia, Charlottesville, VA, USA}
\email{saa6uy@virginia.edu}

\begin{abstract}
For $\FR{A}$ a separable unital $C^*$-algebra and $M$ a separable McDuff II$_1$-factor, we show that the space $\HOM_w(\FR{A},M)$ of weak approximate unitary equivalence classes of unital $*$-homomorphisms $\FR{A}\rightarrow M$ may be considered as a closed, bounded, convex subset of a separable Banach space -- a variation on N. Brown's convex structure $\HOM(N,R^\mathcal{U})$.  When $\FR{A}$ is nuclear, $\HOM_w(\FR{A},M)$ is affinely homeomorphic to the trace space of $\FR{A}$, but in general $\HOM_w(\FR{A},M)$ and the trace space of $\FR{A}$ do not share the same data (several examples are provided).  We characterize extreme points of $\HOM_w(\FR{A},M)$ in the case where either $\FR{A}$ or $M$ is amenable; and we give two different conditions -- one necessary and the other sufficient -- for extremality in general.  The universality of $C^*(\mathbb{F}_\infty)$ is reflected in the fact that for any unital separable $\FR{A}, \HOM_w(\FR{A},M)$ may be embedded as a face in $\HOM_w(C^*(\mathbb{F}_\infty),M)$.  We also extend Brown's construction to apply more generally to $\HOM(\FR{A},M^\mathcal{U})$.

\smallskip

\noindent\textbf{Keywords.} $C^*$-algebra, II$_1$-factor, convex, approximate unitary equivalence, trace, ultrapower.

\end{abstract}

\maketitle

\section{Introduction}

In this paper we will introduce and investigate a convex structure on the space $\HOM_w(\FR{A},M)$ of equivalence classes of $*$-homomorphisms from a unital separable $C^*$-algebra $\FR{A}$ to a separable McDuff II$_1$-factor $M$.  
%The structure of McDuff factors allows for a natural definition of convex combinations in our setting.  
Placing a tractable structure on equivalence classes of homomorphisms between operator algebras is no new idea (e.g. Ext$(\FR{A})$).  In fact, N. Brown presented a convex structure on the typically \emph{nonseparable} space $\HOM(N,R^\mathcal{U})$ of unitary equivalence classes of $*$-homomorphisms from a separable II$_1$-factor $N$ into the ultrapower of the separable hyperfinite II$_1$-factor $R^\mathcal{U}$ in \cite{topdyn}.  In the present paper, we extend the scope to $C^*$-algebras and 
replace the approximation mechanism of Brown's construction -- the ultrapower -- with the mechanism of weak approximate unitary equivalence, 
%shift the approximation mechanism from the target algebra to the equivalence relation of weak approximate unitary equivalence,
allowing us to consider separable target algebras. The result is a \emph{separable} adaptation, $\HOM_w(\FR{A},M)$, of Brown's $\HOM(N,R^\mathcal{U})$ that still retains a convex structure.  We also exhibit a convex structure on a generalization, $\HOM(\FR{A},M^\mathcal{U})$, of Brown's $\HOM(N, R^\mathcal{U})$.  There are interesting connections (and disconnections) between algebraic concepts (e.g. traces, ideals, commutants) and concepts associated with convex geometry (e.g. affine maps, faces, extreme points).  These interactions are explored through both general theorems and specific examples.  

\begin{blankdfn}Given a unital separable $C^*$-algebra $\FR{A}$ and a separable II$_1$-factor $N$, two unital $*$-homomorphisms $\pi, \rho: \FR{A} \rightarrow N$ are \emph{weakly approximately unitarily equivalent} -- denoted $\pi \sim \rho$ -- if there is a sequence of unitaries $\left\{ u_n \right\} \subset \mathcal{U}(N)$ such that for every $a \in \FR{A}$, \[\lim_{n\rightarrow \infty} ||\pi(a) - u_n \rho(a) u_n^*||_2 = 0\] where $||x||_2^2 = \tau(x^*x)$ for $\tau$ the unique tracial state on $N$ (see \cite{voiculescu}, \cite{ding-hadwin}, and \cite{orbits}).  It will be useful to keep the following equivalent formulation of this definition in mind. For $\pi,\rho: \FR{A}\rightarrow N$, $\pi \sim \rho$ if and only if for every finite subset $F \subset \FR{A}_{\leq 1}$ and every $\varepsilon > 0$ there is a unitary $u \in \mathcal{U}(N)$ such that \[||\pi(a) - u\rho(a)u^*||_2 < \varepsilon\] for every $a \in F$.
\end{blankdfn}

\begin{blankdfn}We let $\HOM_w(\FR{A},N)$ denote the space of unital $*$-homomorphisms $\FR{A} \rightarrow N$ modulo the equivalence relation of weak approximate unitary equivalence, and we let $[\pi]$ denote the equivalence class in $\HOM_w(\FR{A},N)$ of $\pi: \FR{A}\rightarrow N$. As we explain below in Definition \ref{metric}, the space $\HOM_w(\FR{A},N)$ can be naturally metrized in a way similar to that of Definition 1.2 of \cite{topdyn}. 
\end{blankdfn}

The foundation of the present paper is the following theorem.

\begin{repthm}{convexlikestructure}
If $M$ is a separable McDuff II$_1$-factor, then $\HOM_w(\FR{A},M)$ may be considered as a closed, bounded, convex subset of a separable Banach space.
\end{repthm} 

\noindent Recall that a II$_1$-factor $M$ is McDuff if and only if $M \cong M \otimes R$, where $R$ denotes the separable hyperfinite II$_1$-factor. We establish the above theorem by showing that $\HOM_w(\FR{A},M)$ satisfies the axioms for a \tql convex-like structure\tqr as in Definition 2.1 of \cite{topdyn}.  In a later paper, \cite{capfri}, it was shown that these axioms characterize a closed, bounded, convex subset of a Banach space. 

It is natural to ask why we restrict to McDuff targets. The main reason is the existence of isomorphisms $\sigma_M: M\otimes R \rightarrow M$ with the following property: \[\sigma_M \circ (\text{id}_M \otimes 1_R) \sim \text{id}_M.\]  Given $\pi: \FR{A} \rightarrow M$, such an isomorphism gives $\sigma_M\circ (\pi \otimes 1_R) \in [\pi]$.  So we can always find a representative whose relative commutant unitally contains a copy of $R$.  As we will see from Definition \ref{convexcombo}, the operation of taking a convex combination of $[\pi_1]$ and $[\pi_2]$ in $\HOM_w(\FR{A},M)$ is obtained by slicing each representative $\sigma_M(\pi_1 \otimes 1_R)$ and $\sigma_M(\pi_2\otimes 1_R)$ by complementary projections of the form $\sigma_M(1_M \otimes p_1)$ and $\sigma_M(1_M \otimes p_2)$ respectively, with both projections contained in both relative commutants.  In this way, the structure of a McDuff factor always provides us with representatives whose relative commutants contain the same copy of $R$ and thus have an interval's worth of projections in common.  Also, allowing any McDuff factor as a target algebra maintains enough generality so that technical embeddability obstructions do not arise.  In fact, requiring a McDuff target is not so much of an obstruction. Thanks to N. Ozawa, we have Theorem \ref{nonmcd} which says that for $\emph{any}$ separable II$_1$-factor $N$ we may consider $\HOM_w(\FR{A},N)$ within this convex context by stabilizing the target algebra to obtain a homeomorphic embedding of $\HOM_w(\FR{A},N)$ as a closed set inside the convex $\HOM_w(\FR{A},N\otimes R)$.
%In particular, no matter how poorly behaved $\FR{A}$ is -- as long as it is finite and unital -- we may find an $M$ such that $\HOM_w(\FR{A},M)$ is not empty.

As mentioned above, $\HOM_w(\FR{A},M)$ is a variation of the object of study \linebreak $\HOM(N,R^\mathcal{U})$  in N. Brown's paper \cite{topdyn}. Here $\HOM(N,R^\mathcal{U})$ denotes the space of $*$-homomorphisms $N\rightarrow R^\mathcal{U}$ from a separable II$_1$-factor $N$ to the ultrapower of $R$ modulo unitary equivalence.  The convex combinations on $\HOM(N,R^\mathcal{U})$ can be viewed as an ultrapower version of the definition for $\HOM_w(\FR{A},M)$ (see \S\ref{ultrasit}).  A major distinction between these two objects is that $\HOM_w(\FR{A},M)$ is always separable (Proposition \ref{separable}) whereas $\HOM(N,R^\mathcal{U})$ is either nonseparable or trivial.  So by studying $\HOM_w(\FR{A},M)$, we have the advantage of studying a separable object.

%One can also notice some similarities shared with the semigroup Ext$(\FR{A})$ (equivalence classes of homomorphisms from the $C^*$-algebra $\FR{A}$ into the Calkin algebra under a slightly stronger notion of unitary equivalence) as in BDF-theory. See \cite{arveson} for some central results.

%In \cite{topdyn}, apart from any operator algebraic framework or language, Brown defined axioms for a \tql convex-like structure\tqr on a metric space and showed that under a certain definition of convex combinations, $\HOM(N,R^\mathcal{U})$ satisfies these axioms.  In a later paper (\cite{capfri}), the authors showed that the axioms prescribed by Brown characterize a closed, bounded, convex subset of a Banach space. We give a definition (\ref{convexcombo}) of convex combinations on $\HOM_w(\FR{A},M)$ that is analagous to the definition of convex combinations in $\HOM(N,R^\mathcal{U})$ and satisfies the axioms of a metric space with convex-like structure.  %Thus, by \cite{capfri}, we will be able to consider $\HOM_w(\FR{A},M)$ as a (separable) closed, bounded, convex subset of a Banach space.

%It is quite natrual to see that $\HOM_w(\FR{A},M)$ is a contravariant functor in the first argument from the category of unital $C^*$-algebras with unital $*$-homomorphisms to the category of affine metric spaces with continuous affine maps; and it is a covariant functor in the second argument from the category of separable II$_1$-factors with unital inclusions to the category of affine metric spaces with continuous of affine metric spaces with continuous affine maps.  

The space $\HOM_w(\FR{A},M)$ has a connection with the trace space of $\FR{A}$. Let $T(\FR{A})$ denote the trace space of $\FR{A}$; recall 
\begin{equation}\label{tracedef}
T(\FR{A}) := \left\{T \in \FR{A}^* | T(1_\FR{A}) = 1, T(a^*a) = T(aa^*)\geq 0 \text{  } \forall a \in \FR{A}\right\}.
\end{equation} 
Let $[\pi] \in \HOM_w(\FR{A},M)$ be the equivalence class of $\pi: \FR{A} \rightarrow M$ and $\tau_M$  denote the unique faithful tracial state on $M$. There is a well-defined, affine map given by 
\begin{align*}
[\pi] & \mapsto  \tau_M \circ \pi 
\end{align*} 
(see Definition \ref{tracecorr}).   We can think of $\tau_M \circ \pi$ as being a trace in $T(\FR{A})$ that \textquotedblleft lifts through $M$,\textquotedblright and injectivity of this map means \textquotedblleft liftable traces remember their homomorphisms (up to weak approximate unitary equivalence).\textquotedblright ~ The following theorem shows that in the nuclear case, this map is very well-behaved.
\begin{repthm}{nuctrace}
If $\FR{A}$ is nuclear, then for any McDuff $M$, $\HOM_w(\FR{A},M)$ is affinely homeomorphic to $T(\FR{A})$ via $[\pi]\mapsto \tau_M\circ \pi$.
\end{repthm}
\noindent In English, this theorem says: for a nuclear $\FR{A}$, all traces on $\FR{A}$ lift through $M$ and remember their homomorphisms.  So in this case $\HOM_w(\FR{A},M)$ serves as a different perspective from which we may study the trace space of $\FR{A}$; and on the other side of the coin, $T(\FR{A})$ gives insight into understanding $\HOM_w(\FR{A},M)$ in general.  We notice how this compares with Ext$(\FR{A})$ -- when $\FR{A}$ is nuclear, we get that Ext$(\FR{A})$ is a group (see \cite{arveson}). Some nontrivial work had to be done to show that there are algebras $\FR{A}$ for which Ext$(\FR{A})$ is not a group (see \cite{anderson} and \cite{haagthor}).  So as in the program of Ext$(\FR{A})$, it is natural to ask: is $\HOM_w(\FR{A},M)$ always the same as $T(\FR{A})$?  We present several examples in \S\S\ref{examples} offering various negative answers (an algebra with forgetful traces, and an algebra with so many traces that they cannot all lift through one $M$).  These examples are encouraging in that they show that the collection of $\left\{\HOM_w(\FR{A},M)\right\}_M$ as $M$ varies over the McDuff factors contains information different from $T(\FR{A})$.  Notice that we would not get much information if we examined this connection in the context of $\HOM(N,R^\mathcal{U})$ because any II$_1$-factor $N$ has a unique tracial state.  Thus, for every $[\pi]  \in \HOM(N,R^\mathcal{U})$ we get $\tau_{R^\mathcal{U}} \circ \pi = \tau_N$.   

We go further to show that the class of algebras $\FR{A}$ for which $\HOM_w(\FR{A},M)$ is affinely homeomorphic to $T(\FR{A})$ is precisely the class of algebras $\FR{A}$ for which given any $T \in T(\FR{A})$, the weak closure of the GNS representation induced by $T$ is hyperfinite -- a class strictly larger than nuclear algebras.  This leads us to a characterization of hyperfiniteness stated in our context of weak approximate unitary equivalence in McDuff factors: a separable, tracial, $R^\mathcal{U}$-embeddable von Neumann algebra $N$ is hyperfinite  if and only if for every separable McDuff II$_1$-factor $M$, any two embeddings $\pi,\rho: N \rightarrow M$ are weakly approximately unitarily equivalent.

%\begin{repthm}{sepchar}
%Let $(N,\sigma)$ be a separable finite tracial von Neumann algebra that embeds into $R^\mathcal{U}$.  $N$ is hyperfinite if and only if for every separable McDuff II$_1$-factor $(M,\tau)$, any two embeddings $\pi,\rho: N \rightarrow M$ are weakly approximately unitarily equivalent.
%\end{repthm}

We turn to consider the convex analysis of $\HOM_w(\FR{A},M)$.  In Proposition 5.2 of \cite{topdyn}, Brown showed that given $[\pi] \in \HOM(N,R^\mathcal{U})$, $[\pi]$ is extreme if and only if $\pi(N)'\cap R^\mathcal{U}$ is a factor. We would like to adapt this characterization to our separable situation.  The analogous statement is not available in our context because the relative commutant of the image of a $*$-homomorphism is not in general well-defined under weak approximate unitary equivalence.  We have the following necessary condition for $[\pi]$ to be extreme.
\begin{repthm}{extfact}
If $[\pi] \in \HOM_w(\FR{A},M)$ is extreme then $W^*(\pi(\FR{A}))$ is a factor.
\end{repthm}
\noindent The converse of the above theorem holds when $\FR{A}$ is nuclear by Theorem \ref{nuctrace}.  Also, when $M = R$, the converse holds for general $\FR{A}$. However, the converse of this theorem fails in general.  

Since our domains are unital separable $C^*$-algebras, we have access to nontrivial ideals. Using the contravariance in the first argument, we show that for a closed two-sided ideal $J$ of $\FR{A}$, $\HOM_w(\FR{A}/J,M)$ is a face of $\HOM_w(\FR{A},M)$.  A statement like this is meaningless in the setting of $\HOM(N,R^\mathcal{U})$ because II$_1$-factors are simple. The observation that any unital separable $C^*$-algebra is a quotient of $C^*(\mathbb{F}_\infty)$ translates into the following surprising fact.
\begin{repthm}{uniface}
For any unital separable $C^*$-algebra $\FR{A}$, $\HOM_w(\FR{A},M)$ is a face of $\HOM_w(C^*(\mathbb{F}_\infty),M)$.
\end{repthm}

%\noindent We notice that the above theorem shows that $\HOM_w(C^*(\mathbb{F}_\infty),M)$ shares a sort of universal property similar to that of the Poulsen simplex (see \cite{LOS}).

We also discuss ultrapowers in considering $\HOM(\FR{A},M^\mathcal{U})$: the space of all unital $*$-homomorphisms $\FR{A}\rightarrow M^\mathcal{U}$ modulo unitary equivalence.  This is an obvious generalization of $\HOM(N,R^\mathcal{U})$ and also supports a convex structure.    We extend Brown's characterization of extreme points to apply to $\HOM(\FR{A},M^\mathcal{U})$: $[\pi] \in \HOM(\FR{A},M^\mathcal{U})$ is extreme if and only if $\pi(\FR{A})'\cap M^\mathcal{U}$ is a factor.  We observe that $\HOM_w(\FR{A},M)$ may be embedded into $\HOM(\FR{A},M^\mathcal{U})$ via $[\pi] \mapsto [\pi^\mathcal{U}]$ where $\pi^\mathcal{U}$ denotes $\pi$ followed by the canonical constant-sequence embedding of $M$ into $M^\mathcal{U}$.  This is a strict inclusion in general, but we observe that $\HOM_w(\FR{A},M) \cong \HOM(\FR{A},M^\mathcal{U})$ in the nuclear case.  This embedding yields the following sufficient condition for extreme points.
\begin{repthm}{comext}
If $\pi^\mathcal{U}(\FR{A})'\cap M^\mathcal{U}$ is a factor, then $[\pi]$ is extreme in $\HOM_w(\FR{A},M)$.
\end{repthm}
\noindent The converse holds in the case when either $\FR{A}$ or $M$ is amenable. It is unknown if the converse of Theorem \ref{comext} holds in general; it would hold if one could show that $\HOM_w(\FR{A},M)$ embeds as a face of $\HOM(\FR{A},M^\mathcal{U})$.  As a consequence of the characterizations of extreme points in the amenable cases, we get an equivalence of two purely algebraic statements with no reference to $\HOM_w(\FR{A},M)$ (see Corollary \ref{nucchar} and Theorem \ref{hyperchar}).  This discussion of relative commutants in ultrapowers along with a helpful comment made by S. White leads us to the following new characterization of the hyperfinite II$_1$-factor.
\begin{repthm}{Rchar}
Let $N$ be an embeddable separable II$_1$-factor.  The following are equivalent:
\begin{enumerate}
\item $N=R$;
\item For any separable II$_1$-factor $X$ and any embedding $\pi: N \rightarrow X^\mathcal{U}$, $\pi(N)'\cap X^\mathcal{U}$ is a factor;
\item For any separable II$_1$-factor $X$ and any embedding $\pi: N\rightarrow X^\mathcal{U}$, the collection of tracial states $\left\{\tau(\pi(x)\cdot): x \in N^+, \tau(x) = 1\right\}$  is weak-$*$ dense in the trace space of $\pi(N)'\cap X^\mathcal{U}$.
\end{enumerate}
\end{repthm} 
\noindent Notice that this is a strengthening of Corollary 5.3 in \cite{topdyn}.

Though we will not be working in the following context, we mention here that in \cite{brocap}, for $N$ a separable II$_1$-factor, Brown and Capraro constructed a real Banach space that naturally contains $\HOM(N, M^\mathcal{U})$.  This space is constructed by applying the Grothendieck construction to a cancellative semigroup structure on the space of $*$-homomorphisms of $N$ into amplifications of $M^\mathcal{U}$.  This Banach space and its dynamics are interesting on their own, but examining such things is not within the scope of the present paper.

This paper is organized as follows.  \textbf{\S\ref{preliminaries} Preliminaries:} We provide all of the initial definitions for $\HOM_w(\FR{A},M)$.  The convex structure of $\HOM_w(\FR{A},M)$ is introduced and verified.  We briefly discuss some surface-level functoriality. \textbf{\S\ref{tracespace} Connection to the Trace Space:}  Here we introduce the relationship between $\HOM_w(\FR{A},M)$ and $T(\FR{A})$ as mentioned above.  We establish the fact that the relationship is a bijection when $\FR{A}$ is nuclear.  We also show that traces remember their homomorphisms when $M = R$.  We then discuss several examples including the \tql forgetful trace\tqr and \tql too many traces\tqr examples mentioned above. \textbf{\S\ref{convexanalysis} Convex Analysis:} We take a closer look at some of the convex analysis of $\HOM_w(\FR{A},M)$.  Theorems \ref{extfact} and \ref{uniface} are the main results of this section. We provide an example showing that the converse of Theorem \ref{extfact} is false in general. \textbf{\S\ref{ultrasit} Ultrapower Situation:} We generalize $\HOM(N,R^\mathcal{U})$ by considering the space $\HOM(\FR{A},M^\mathcal{U})$.  We extend the characterization of extreme points in $\HOM(N,R^\mathcal{U})$ to a characterization in $\HOM(\FR{A},M^\mathcal{U})$.  We prove Theorem \ref{Rchar}.  The embedding $\HOM_w(\FR{A},M) \subset \HOM(\FR{A},M^\mathcal{U})$ is discussed.  This yields a sufficient condition for extreme points in $\HOM_w(\FR{A},M)$ and characterizations in the cases where either $\FR{A}$ is nuclear or $M$ is hyperfinite.  \textbf{\S\ref{stabilization} Stabilization:}  
%We discuss continuity in the first argument of $\HOM_w(\FR{A},M)$.  It is shown that when a coassociative comultiplication on $\FR{A}$ exists (e.g. $\FR{A}$ a quantum group), we can define an affine-distributive, associative product on $\HOM_w(\FR{A},R)$. 
In this section we address the stabilization of non-McDuff target algebras.  We show in Theorem \ref{nonmcd} that for any separable II$_1$-factors $N_1$ and $N_2$, $\HOM_w(\FR{A},N_1)$ embeds homeomorphically into $\HOM_w(\FR{A}, N_1\otimes N_2)$ as a closed subset.  So in particular, $\HOM_w(\FR{A},N) \subset \HOM_w(\FR{A},N \otimes R)$.
%We hope to broaden this construction to be applicable to maps between $C^*$-algebras modulo approximate 2-colored equivalence in the sense of \cite{bbstww}. 
%Lastly, some thoughts on the existence of extreme points in $\HOM(N,R^\mathcal{U})$ are discussed.
% and lastly we make some remarks on P\u{a}unescu's convex structure on Sofic embeddings into ultrapowers of matrix permuations as in \cite{paunescu}.
\textbf{\S\ref{oq} Open Questions:} We conclude the paper with some related open problems.

\subsection*{Acknowledgements}  The author would like to express his gratitude to his Ph.D. supervisor David Sherman for his guidance.  The author thanks Taka Ozawa for suggesting the proof of Theorem \ref{presineq} and Stuart White for suggesting the proof of Theorem \ref{stuart}. The author would also like to thank Nate Brown and Jesse Peterson for their helpful comments and suggestions regarding this project.  Lastly, the author thanks the many mathematicians not listed here who provided helpful feedback and encouragement.

%Thanks to  Dietmar Bisch, Nate Brown, Ben Hayes, Vaughan Jones, Mike Jury, Victor Kaftal, Elias Katsoulis, Craig Kleski, Kostya Medynets, Jesse Peterson, Chris Ramsey, David Sherman, Hannes Thiel, Vrej Zarikian, ... Sentences to be filled in.

\section{Preliminaries}\label{preliminaries}

Unless otherwise noted, $\FR{A}$ will denote a separable unital $C^*$-algebra, $N$ will denote a separable II$_1$-factor, $M$ will denote a separable McDuff II$_1$-factor (see \cite{mcduff}), and $R$ will denote the separable hyperfinite II$_1$-factor.

%\begin{dfn}\label{wauedef}Two $*$-homomorphisms $\pi,\rho: \mathfrak{A} \rightarrow N$ are \emph{weakly approximately unitarily equivalent in $N$} (denoted $\pi \sim_N \rho$) if there is a net of unitaries $\left\{u_n\right\} \subset \mathcal{U}(N)$ such that for every $a \in \FR{A}$, \[\lim_n ||\pi(a) - u_n \rho(a) u_n^*||_2=0\]  where $||x||_2^2 = \tau(x^*x)$ for $\tau$ the unique tracial state on $N$. When no confusion can occur, we simply say $\pi$ and $\rho$ are weakly approximately unitarily equivalent (denoted $\pi \sim \rho$).  This is indeed an equivalence relation.
%\end{dfn}

%\begin{dfn}
%Let $\HOM_w(\FR{A},N)$ denote the set of $*$-homomorphisms from $\FR{A}$ to $N$ modulo the equivalence relation of weak approximate unitary equivalence.  Let $[\pi]$ denote the equivalence class of $\pi: \FR{A}\rightarrow N$.
%\end{dfn}

We consider the topology of pointwise convergence (of equivalence classes) for $\HOM_w(\mathfrak{A},N)$.  That is, for $[\pi_n],[\pi] \in \HOM_w(\FR{A},N),$ $ [\pi_n] \rightarrow [\pi]$ if there are representatives $\pi_n' \in [\pi_n]$ such that $\pi_n'(a) \rightarrow \pi(a)$ under the $||\cdot||_2$-norm for every $a \in \FR{A}$.  This topology can be metrized in the following way.

\begin{dfn}\label{metric}
For $[\pi], [\rho] \in \HOM_w(\mathfrak{A}, N)$, let $\left\{a_n\right\}$ be a countable generating set in $\mathfrak{A}_{\leq1}$ and define the metric (same as in Definition 1.2 of \cite{topdyn}) \[d([\pi],[\rho]) = \inf_{u \in \mathcal{U}(N)} \Big(\sum_{n=1}^{\infty} \frac{1}{2^{2n}} || \pi(a_n) - u \rho(a_n)u^*||^2_2\Big)^{\frac{1}{2}}.\]
\end{dfn}

\noindent This is quickly seen to be a metric that induces the topology described above. We note that the objects of study in \cite{topdyn} are typically not second countable with respect to the corresponding metric, but in our situation we have the following fact.

\begin{prop}\label{separable}
$\HOM_w(\FR{A},N)$ is complete and separable under the metric $d$.
\end{prop}

\begin{proof}

Completeness follows from an argument identical to one found in the proof of Proposition 4.6 of \cite{topdyn}.

For separability under $d$, let $(N_{\leq 1})^{\mathbb{N}}/\sim$ denote the set of all sequences in the unit ball of $N$ modulo the equivalence relation given by $\left\{x_n\right\} \sim \left\{y_n\right\}$ if there is a sequence of unitaries $\left\{u_p\right\} \subset \mathcal{U}(N)$ such that for every $n$ we have $x_n = \lim_p u_py_n u_p^*$ where the limit is taken in the $||\cdot||_2$-norm. Let $[\left\{x_n\right\}]$ denote the equivalence class of $\left\{x_n\right\}$ under this equivalence relation. Consider a metric $d'$ on $(N_{\leq 1})^{\mathbb{N}}/\sim$ given by \[d'([\left\{x_n\right\}],[\left\{y_n\right\}]) = \inf_{u \in \mathcal{U}(N)} \Big(\sum_{n=1}^{\infty} \frac{1}{2^{2n}} || x_n - u y_nu^*||^2_2\Big)^{\frac{1}{2}}.\]

We claim that $(N_{\leq 1})^{\mathbb{N}}/\sim$ is separable under $d'$. 

Let $\left\{m_n\right\}$ be $||\cdot||_2$-dense in $N_{\leq1}$.  Fix $\varepsilon > 0$ and $[\left\{x_n\right\}] \in (N_{\leq 1})^{\mathbb{N}}/\sim$.  Let $K \in \mathbb{N}$ be such that \[\displaystyle 4\sum_{n= K+1}^{\infty} \frac{1}{2^{2n}} < \frac{\varepsilon^2}{2},\]  and let $f: \left\{1, \dots, K\right\} \rightarrow \mathbb{N}$ be such that \[\frac{1}{2^{2n}}||x_n - m_{f(n)}||_2^2 < \frac{\varepsilon^2}{2K}, \forall 1 \leq n \leq K.\] For such a $K$ and $f$, put $\left\{z_{K,f,n}\right\} \subset N_{\leq 1}$ with \[z_{K,f,n}= \left\{\begin{array}{lcr}
 m_{f(n)} & \text{if} & 1 \leq n \leq K \\ 
m_{n+K'} & \text{if} & n>K 
\end{array}\right.\] where $K' = \max\left\{f(n) : 1\leq n \leq K\right\}$.

Then we have
\begin{align*}
d'([\left\{x_n\right\}],[\left\{z_{K,f,n}\right\}])^2 & \leq  \sum_{n=1}^{\infty} \frac{1}{2^{2n}}||x_n - z_{K,f,n}||_2^2\\
 & =  \sum_{n=1}^K \frac{1}{2^{2n}}||x_n - z_{K,f,n}||_2^2+\sum_{n=K+1}^{\infty}\frac{1}{2^{2n}}||x_n - z_{K,f,n}||_2^2\\
 & <  K\cdot\frac{\varepsilon^2}{2K} + \frac{\varepsilon^2}{2} = \varepsilon^2.
\end{align*}

Thus $d'([\left\{x_n\right\}],[\left\{z_{K,f,n}\right\}]) < \varepsilon$.

So \[ \left\{ [\left\{z_{K,f,n}\right\}] | K \in \mathbb{N}, f: \left\{1,\dots, K\right\} \rightarrow \mathbb{N}\right\} = \bigcup_{K=1}^{\infty} \bigcup_{f:\left\{1,\dots, K\right\} \rightarrow \mathbb{N}} \left\{[\left\{z_{K,f,n}\right\}]\right\}\] is dense and countable.  Thus $(N_{\leq 1})^{\mathbb{N}}/\sim$ is separable under the metric $d'$.

By fixing a generating sequence $\left\{a_n\right\}$ in $\FR{A}_{\leq 1}$, we get the metric $d$ on $\HOM_w(\FR{A},N)$ as defined above.  We can consider the metric space  $(\HOM_w(\FR{A},N),d)$ as a subspace of the metric space $((N_{\leq 1})^{\mathbb{N}}/\sim,d')$ by identifying $[\pi] \in \HOM_w(\FR{A},N)$ with $[\left\{\pi(a_n)\right\}] \in (N_{\leq 1})^{\mathbb{N}}/\sim$.  Since subspaces of separable metric spaces are separable, the proof is complete.\end{proof}

\begin{rmk}\label{metricgenerators}
This metric is not canonical--it depends on the choice of the generating sequence.  It will sometimes be useful to choose our generating sequence to be a sequence of unitaries (always possible in a unital $C^*$-algebra)--see Theorem \ref{nonmcd}.
\end{rmk}

\noindent We will see later in Example \ref{surjfail} that $\HOM_w(\FR{A},M)$ is not necessarily compact.

\subsection{Convex Structure} We now turn to define a convex structure on \linebreak $\HOM_w(\FR{A},M)$ for $M$ McDuff.  In \cite{topdyn} Brown uses certain isomorphisms between corner algebras $pR^{\mathcal{U}}p$ and $R^{\mathcal{U}}$ to define a convex structure.  We take a slightly different approach in order to define convex combinations in $\HOM_w(\FR{A},M)$.  We must first introduce some terminology.

\begin{dfn}\label{regiso}
For a McDuff II$_1$-factor, a \emph{regular isomorphism} $\sigma: M \otimes R \rightarrow M$ is an isomorphism such that $\sigma\circ (\text{id}_M \otimes 1_R) \sim \text{id}_M$ where $(\text{id}_M\otimes 1_R)(x) = x \otimes 1_R$ for $x \in M$.  Denote the set of regular isomorphisms of $M$ as $\text{REG}(M)$. 
\end{dfn}

\begin{prop}\label{pitensor1}
Let $M$ be a McDuff II$_1$-factor.
\begin{enumerate}

\item REG$(M) \neq \emptyset$.

\item Any two regular isomorphisms $\sigma_M,s_M: M \otimes R \rightarrow M$ are weakly approximately unitarily equivalent.

\item \label{regaut} The following are equivalent.
	\begin{enumerate}
		\item For every isomorphism $\nu: M\otimes R \rightarrow M, \nu \in \text{REG}(M);$

		\item $\overline{\text{Inn}(M)} = \text{Aut}(M)$. (The closure is in the point-$||\cdot||_2$ topology).

	\end{enumerate}

\end{enumerate}
\end{prop}

\begin{proof}
\noindent (1): We will construct an isomorphism $\sigma_M: M\otimes R \rightarrow M$ such that $\text{id}_M \sim \sigma_M(\text{id}_M \otimes 1_R)$.  Let $\nu: M\otimes R \rightarrow M$ and $\epsilon: R\otimes R \rightarrow R$ be isomorphisms. It is a well-known fact that any unital endomorphism of $R$ is approximately inner.  We apply this fact to the map $\epsilon \circ (\text{id}_R \otimes 1_R):R \rightarrow R$ getting that $\epsilon \circ(\text{id}_R \otimes 1_R) \sim \text{id}_R$.  Let \[\sigma_M := \nu \circ (\text{id}_M \otimes \epsilon) \circ (\nu^{-1} \otimes \text{id}_R)\] and consider
\begin{align*}
\sigma_M(\text{id}_M \otimes 1_R) &= \nu \circ (\text{id}_M \otimes \epsilon) \circ (\nu^{-1} \otimes \text{id}_R) \circ (\text{id}_M \otimes 1_R)\\
& =  \nu \circ (\text{id}_M \otimes \epsilon) \circ (\nu^{-1} \otimes 1_R)\\
& =  \nu \circ (\text{id}_M \otimes \epsilon) \circ (\text{id}_M \otimes \text{id}_R \otimes \text{id}_R) \circ (\nu^{-1} \otimes 1_R)\\
& =  \nu \circ (\text{id}_M \otimes \epsilon) \circ (\text{id}_M \otimes \text{id}_R \otimes 1_R) \circ \nu^{-1}\\
& =  \nu \circ (\text{id}_M \otimes (\epsilon \circ (\text{id}_R \otimes 1_R))) \circ \nu^{-1}\\
& \sim  \nu \circ (\text{id}_M \otimes \text{id}_R) \circ \nu^{-1}\\
& =  \text{id}_M.
\end{align*}

\noindent (2): From Definition \ref{regiso} we have that $\sigma_M^{-1} \sim \text{id}_M \otimes 1_R \sim s_M^{-1}$.  Then it is a straight-forward exercise to see that this implies that $\sigma_M \sim s_M$.

\noindent (3): ($a\Rightarrow b$): Let $\alpha \in \text{Aut}(M)$, and let $\nu: M\otimes R \rightarrow M$ be an isomorphism.  Define $\nu_{\alpha} := \alpha \circ \nu$.  By assumption and part (2), $\nu \sim \nu_{\alpha}$, or equivalently, $\nu^{-1} \sim \nu_{\alpha}^{-1}$.  So we get 
\begin{align*}
\alpha & =  \alpha \circ \nu \circ \nu^{-1}\\
& =  \nu_{\alpha} \circ \nu^{-1}\\
& \sim  \nu_{\alpha} \circ \nu_{\alpha}^{-1}\\
& =  \text{id}_M.
\end{align*}

\noindent ($b\Rightarrow a$): Let $\nu: M\otimes R \rightarrow M$ be an isomorphism, and let $\sigma \in \text{REG}(M)$.  Then $\nu \circ \sigma^{-1} \in\text{Aut}(M)$.  Thus 
\begin{align*}
\nu\circ \sigma^{-1} \sim \text{id}_M &\Rightarrow \nu^{-1} \sim \sigma^{-1}\\
&\Rightarrow \nu^{-1} \sim \text{id}_M \otimes 1_R\\
& \Rightarrow \nu \in \text{REG}(M).\qedhere
\end{align*}

\end{proof}

\begin{rmk}\label{sakairmk} In Theorem 8 of \cite{sakai1}, Sakai gave an example of a McDuff factor ($\displaystyle \otimes_{\mathbb{Z}} L(\mathbb{F}_2)$) that fails condition (b) in statement (3) of Proposition \ref{pitensor1}. So it is nontrivial for us to restrict to regular isomorphisms in this paper.
\end{rmk}

For $\pi: \FR{A} \rightarrow M$ and $p$ a projection in $R$, let $\pi \otimes p: \FR{A} \rightarrow M \otimes R$ be given by $(\pi \otimes p)(a) = \pi(a) \otimes p$. We now define convex combinations in $\HOM_w(\FR{A},M)$.

\begin{dfn}\label{convexcombo}
Given $[\pi],[\rho] \in \HOM_w(\FR{A},M)$ and $t \in [0,1]$ we define 
\begin{align}
t[\pi] + (1-t)[\rho] &:= [\sigma_M(\pi \otimes p) + \sigma_M(\rho \otimes p^{\perp})]\label{convexdef}
\end{align} where $p \in \mathcal{P}(R)$ with $\tau(p) = t$ and $\sigma_M: M\otimes R \rightarrow M$ is a regular isomorphism.
\end{dfn}

\noindent Compare Definition \ref{convexcombo} with the discussion in Example 4.5 of \cite{topdyn}.  Clearly, this definition extends to taking convex combinations of $n$ equivalence classes. The following picture is helpful in visualizing this operation.

\begin{align*}
t[\pi] + (1-t)[\rho] \mapsto&
\left[\sigma_M\left( \begin{array}{c|c} 
\pi \otimes p & 0 \\
\hline
0 & \rho \otimes p^{\perp}  
\end{array}\right)\right]
\end{align*}
where the $2\times 2$ matrix corresponds to the decomposition via $1_M \otimes p$ and $1_M \otimes p^\perp$.

\begin{prop}\label{welldefined}
The formula \eqref{convexdef} is well-defined.  That is, for $\sigma_M$ and $s_M$ regular isomorphisms, $p, q \in \mathcal{P}(R)$ with $\tau(p) = \tau(q) = t$, and $[\pi_1] = [\pi_2], [\rho_1] = [\rho_2]$, then \[ [\sigma_M(\pi_1 \otimes p) + \sigma_M(\rho_1 \otimes p^{\perp})] = [s_M(\pi_2 \otimes q) + s_M(\rho_2 \otimes q^{\perp})].\]
\end{prop}

\begin{proof}
By Proposition \ref{pitensor1} (2) we have 
\begin{align}
[\sigma_M(\pi_2 \otimes q) + \sigma_M(\rho_2 \otimes q^{\perp})] &= [s_M(\pi_2 \otimes q) + s_M(\rho_2 \otimes q^{\perp})].\label{wd1}
\end{align}

Let $v,w \in R$ be partial isometries such that 
\begin{align*}
v^*v &=  p,  & vv^*  &=   q,\\
w^*w  &=  p^{\perp}, & ww^*  &=  q^{\perp}.
\end{align*}
Then $u := \sigma_M(1_M \otimes (v+w))$ is a unitary with
\begin{align}
\sigma_M(\pi_2 \otimes p) + \sigma_M(\rho_2 \otimes p^{\perp}) &= u^*(\sigma_M(\pi_2 \otimes q) + \sigma_M(\rho_2 \otimes q^{\perp}))u. \label{wd2}
\end{align}

Let $\left\{u_n\right\}, \left\{v_n\right\} \subset \mathcal{U}(M)$ be such that
\begin{align*}
\pi_1(a) & =  \lim_n u_n \pi_2(a) u_n^*,\\
\rho_1(a) & =  \lim_n v_n \rho_2(a) v_n^*
\end{align*}
for every $a \in \mathfrak{A}$.  Let $w_n = \sigma_M(u_n \otimes p) + \sigma_M(v_n \otimes p^{\perp}).$  Then $\left\{w_n\right\} \subset \mathcal{U}(M)$ with
\begin{align}
\sigma_M(\pi_1 \otimes p) + \sigma_M(\rho_1 \otimes p^{\perp}) &= \lim_n w_n(\sigma_M(\pi_2\otimes p) + \sigma_M(\rho_2 \otimes p^{\perp}))w_n^*. \label{wd3}
\end{align}

So by \eqref{wd3}, \eqref{wd2}, and \eqref{wd1} respectively we have
\begin{align*}
[\sigma_M(\pi_1 \otimes p) + \sigma_M(\rho_1 \otimes p^{\perp}] & =  [\sigma_M(\pi_2 \otimes p) + \sigma_M(\rho_2 \otimes p^{\perp})]\\
& =  [\sigma_M(\pi_2 \otimes q) + \sigma_M(\rho_2 \otimes q^{\perp})]\\
& =  [s_M(\pi_2 \otimes q) + s_M(\rho_2 \otimes q^{\perp})].\qedhere
\end{align*}
\end{proof}

We are now ready to prove the following theorem.

\begin{thm}\label{convexlikestructure}
With convex combinations defined as in Definition \ref{convexcombo}, $\HOM_w(\FR{A},M)$ satisfies the axioms of Brown's convex-like-structure (Definition 2.1 of \cite{topdyn}).  Therefore, by Proposition \ref{separable} and the main result of \cite{capfri}, $\HOM_w(\FR{A},M)$ may be considered as a closed, bounded, convex subset of a separable Banach space.
\end{thm}

\begin{proof}

Given $[\pi_1], \dots, [\pi_n] \in \HOM_w(\FR{A},M)$ and $0 \leq t_1,\dots t_n \leq 1$ such that $\sum t_i = 1$, we must show that $\HOM_w(\FR{A},M)$ is complete under $d$ and that Definition \ref{convexcombo} satisfies the following axioms.

\begin{enumerate}

	\item (commutativity) $t_1[\pi_1] + \cdots t_n[\pi_n] = t_{\alpha(1)}[\pi_{\alpha(1)}] \cdots t_{\alpha(n)}[\pi_{\alpha(n)}]$ for every permutation $\alpha \in S_n$.

	\item (linearity) if $[\pi_1] = [\pi_2]$  then $t_1[\pi_1] + t_2[\pi_2] + t_3[\pi_3] + \cdots + t_n[\pi_n] = (t_1+t_2)[\pi_1] + t_3[\pi_3] + \cdots t_n [\pi_n]$.

	\item (scalar identity) if $t_i = 1$ then $t_1[\pi_1] + \cdots + t_n[\pi_n] = [\pi_i]$.

	\item (metric compatibility) $d((t_1[\pi_1] + \cdots + t_n[\pi_n]), (t_1'[\pi_1] + \cdots t_n'[\pi_n])) \leq C \sum |t_i - t_i'|$ and $d((t_1[\pi_1] + \cdots + t_n[\pi_n]),(t_1[\pi_1'] + \cdots + t_n[\pi_n'])) \leq \sum t_i d([\pi_i],[\pi_i'])$.

	\item (algebraic compatibility) \[s\Big(\sum_{i=1}^n t_i[\pi_i]\Big) + (1-s)\Big(\sum_{j=1}^m t_j'[\pi_j']\Big) = \sum_{i=1}^n st_i[\pi_i] + \sum_{j=1}^m (1-s)t_j'[\pi_j'].\]  

\end{enumerate}

We have completeness by Proposition \ref{separable}. Metric compatibility follows from an argument identical to the one found in Proposition 4.6 of \cite{topdyn}.

Commutativity and scalar identity are automatic.

We check linearity.  That is, if $[\pi_1] = [\pi_2]$ then \[t_1 [\pi_1] + t_2[\pi_2] + \cdots + t_n[\pi_n] = (t_1+t_2)[\pi_1] + \cdots + t_n[\pi_n].\]  By definition, for $\sigma_M$ a regular isomorphism, we have that 
\begin{align*}
t_1[\pi_1] + t_2[\pi_1] + \cdots + t_n[\pi_n] &=  [\sigma_M(\pi_1 \otimes p_1) + \sigma_M(\pi_1 \otimes p_2) + \cdots + \pi_n \otimes p_n]\\
& =  [\sigma_M(\pi_1 \otimes (p_1 + p_2)) + \cdots + \pi_n \otimes p_n]\\
& =  (t_1 + t_2)[\pi_1] + \cdots + t_n[\pi_n]. 
\end{align*}

We next check algebraic compatibility.  That is, for $0\leq t_i,t_j',s \leq 1$ with $\sum t_i  = \sum t_j'=1$, then \[s\Big(\sum t_i [\pi_i]\Big) + (1-s)\Big(\sum t_j'[\pi_j']\Big) = \sum st_i[\pi_i]  + \sum(1-s)t_j'[\pi_j'].\]
We have that 
\[s\Big(\sum t_i [\pi_i]\Big) + (1-s)\Big(\sum t_j'[\pi_j']\Big) \]
\begin{align*}
& =  s\Big[\sum \sigma_M(\pi_i \otimes p_i)\Big] + (1-s)\Big[\sum \sigma_M(\pi_j' \otimes p_j')\Big]\\
& =  \Big[\sum \sigma_M(\sigma_M(\pi_i \otimes p_i) \otimes p_s) + \sum \sigma_M(\sigma_M(\pi_j' \otimes p_j') \otimes p_{1-s})\Big]\\
& =  \Big[\sigma_M((\sigma_M \otimes \text{id}_R)\Big(\sum \pi_i \otimes p_i \otimes p_s + \sum \pi_j' \otimes p_j' \otimes  p_{1-s}\Big)\Big]
\end{align*}
for projections $p_s,p_{1-s},p_i,p_j' \in \mathcal{P}(R)$ with $\tau(p_s) =s, \tau(p_{1-s}), \tau(p_i) = t_i, \tau(p_j') = t_j'$ with $p_{1-s} = 1 - p_s$, $p_i$s pairwise orthogonal, and the $p_j'$s pairwise orthogonal.

From Definition \ref{regiso} we know that $\sigma_M^{-1}$ is weakly approximately unitarily equivalent to $\text{id}_M\otimes 1_R$.  Also $\epsilon\circ (1_R \otimes \text{id}_R)$ is weakly approximately unitarily equivalent to $\text{id}_R$ because it is a unital endomorphism of $R$.  So we get the following equivalences with respect to weak approximate unitary equivalence.

\begin{align*}
(\text{id}_M \otimes \epsilon) \circ (\sigma_M^{-1} \otimes \text{id}_R) & \sim  (\text{id}_M \otimes \epsilon) \circ (\text{id}_M \otimes 1_R \otimes \text{id}_R)\\
& =  \text{id}_M \otimes (\epsilon \circ (1_R \otimes \text{id}_R))\\
& \sim  \text{id}_M \otimes \text{id}_R \\
& = \text{id}_{M\otimes R}
\end{align*}

It follows that 
\begin{align*}
(\text{id}_M \otimes \epsilon) &\sim (\sigma_M \otimes \text{id}_R).
\end{align*}

Thus we get that 
\[s\Big(\sum t_i [\pi_i]\Big) + (1-s)\Big(\sum t_j'[\pi_j']\Big)\]
\begin{align*}
  =& \quad  \Big[\sigma_M\Big((\sigma_M \otimes \text{id}_R)\Big(\sum \pi_i \otimes p_i \otimes p_s + \sum \pi_j' \otimes p_j' \otimes p_{1-s}\Big)\Big)\Big]\\
 = & \quad \Big[\sigma_M\Big((\text{id}_M \otimes \epsilon)\Big(\sum \pi_i \otimes p_i \otimes p_s + \sum \pi_j' \otimes p_j' \otimes p_{1-s}\Big)\Big)\Big]\\
 = & \quad \Big[\sum \sigma_M(\pi_i \otimes \epsilon(p_i \otimes p_s)) + \sum \sigma_M(\pi_j' \otimes \epsilon(p_j' \otimes p_{1-s}))\Big]\\
 = & \quad \sum st_i[\pi_i]  + \sum(1-s)t_j'[\pi_j']
\end{align*}
since the operation is well-defined.\end{proof}

\subsection{Functoriality}\label{functoriality}

A $*$-homomorphism $\varphi: \mathfrak{A} \rightarrow \mathfrak{B}$ induces an affine map $\varphi^*: \HOM_w(\mathfrak{B}, M) \rightarrow \HOM_w(\mathfrak{A},M)$ given by \[\varphi^*([\pi]) = [\pi\circ \varphi].\]

\begin{prop}\label{phihat}
The induced map $\varphi^*$ is well-defined, continuous, and affine.
\end{prop}

\begin{proof}
Well-Defined:  Let $[\pi] = [\rho] \in \HOM_w(\mathfrak{B},M)$.  So there is a sequence of unitaries $\left\{u_n\right\} \subset M$ such that $\rho(b) = \lim_n u_n \pi(b) u_n^*$ for any $b \in \mathfrak{B}$.  Thus, for $a \in \mathfrak{A}$, we have that $\rho(\varphi(a)) = \lim_n u_n \pi(\varphi(a)) u_n^*$.  Therefore $\varphi^*([\pi]) = [\pi \circ \varphi] = [\rho \circ \varphi] = \varphi^*([\rho])$.  Continuity is just as quick to see.

Affine: \begin{align*}
\varphi^*(t[\pi] + (1-t)[\rho]) & =  \varphi^*([\sigma_M(\pi \otimes p) + \sigma_M(\rho \otimes p^{\perp})])\\
& =  [(\sigma_M(\pi \otimes p) + \sigma_M(\rho \otimes p^{\perp}))\circ \varphi] \\
& =  [\sigma_M((\pi \circ \varphi )\otimes p) + \sigma_M((\rho\circ \varphi) \otimes p^{\perp})]\\
& =  t[\pi \circ \varphi] + (1-t)[\rho\circ \varphi]\\
& =  t\varphi^*([\pi]) + (1-t)\varphi^*([\rho]).\qedhere
\end{align*}
\end{proof}

\noindent The chain rule and preservation of identity are obvious observations; so we see that $\HOM_w(\cdot,M)$ is a contravariant functor from the category of $C^*$-algebras to the category of affine metric spaces.  

\begin{exmpl}
We exhibit an injective homomorphism $\varphi$ such that $\varphi^*$ fails to be surjective.  Consider $\varphi: \mathbb{C} \oplus \mathbb{M}_2 \rightarrow \mathbb{C} \oplus \mathbb{M}_3$ given by 
\[\lambda \oplus \left(\begin{matrix} a & b \\ c& d \end{matrix}\right) \mapsto \lambda \oplus \left(\begin{matrix} a & b & 0 \\ c & d & 0 \\ 0 & 0 & \lambda\end{matrix}\right). \]
It is well-known that up to unitary equivalence, $*$-homomorphisms from finite dimensional algebras into II$_1$ factors are completely determined by their induced traces (see also Theorem \ref{nuctrace}).  Therefore the induced map $\varphi^*: \HOM_w( \mathbb{C} \oplus \mathbb{M}_3, M) \rightarrow \HOM_w(\mathbb{C} \oplus \mathbb{M}_2,M)$ may be understood as $\varphi^*: T( \mathbb{C} \oplus \mathbb{M}_3) \rightarrow T(\mathbb{C} \oplus \mathbb{M}_2)$ where given $f \in T( \mathbb{C} \oplus \mathbb{M}_3)$ we have $\varphi^*(f) = f \circ \varphi$.  Since both algebras have two summands, both trace spaces are the two-vertex simplex (i.e. the unit interval).  The fact that $\varphi^*$ is affine allows us to only check the images of the extreme points (endpoints) under $\varphi^*$ in order to see the image $\varphi^*(T(\mathbb{C} \oplus \mathbb{M}_3))$.  One endpoint of $T( \mathbb{C} \oplus \mathbb{M}_3)$ is the trace 
\[f_1(\lambda \oplus (a_{ij}))= \lambda.\]  We see that \[\varphi^*(f_1)\Bigg(\lambda \oplus \left(\begin{matrix} a & b \\ c& d \end{matrix}\right)\Bigg) = \lambda.\]  The other endpoint of $T( \mathbb{C} \oplus \mathbb{M}_3)$ is 
\[f_2(\lambda \oplus (a_{ij})) = \text{tr}_3(a_{ij})\] where $\text{tr}_3$ is the (unique) normalized trace on $\mathbb{M}_3$.  We get that \[\varphi^*(f_2)\Bigg(\lambda \oplus \left(\begin{matrix} a & b \\ c& d \end{matrix}\right)\Bigg) = \frac{1}{3}\lambda + \frac{2}{3}\text{tr}_2\Bigg(\left(\begin{matrix} a & b \\ c& d \end{matrix}\right)\Bigg)\] where $\text{tr}_2$ is the normalized trace on $\mathbb{M}_2$.  So the image $\varphi^*(T(\mathbb{C} \oplus \mathbb{M}_3))$ is the convex hull of $\varphi^*(f_1)$ and $\varphi^*(f_2)$.  From this it is clear that the extreme trace \[\lambda \oplus    \left(\begin{matrix} a & b \\ c& d \end{matrix}\right) \mapsto \text{tr}_2\Bigg(\left(\begin{matrix} a & b \\ c& d \end{matrix}\right)\Bigg)\] does not lie in the image $\varphi^*(T(\mathbb{C} \oplus \mathbb{M}_3))$.  Hence $\varphi^*$ is not surjective.
\end{exmpl}

On the other hand, we have the following fact.

\begin{prop}\label{surin}
If $\varphi$ is surjective, then $\varphi^*$ is an affine homeomorphism onto its image.
\end{prop}

\begin{proof}
We must show that $\varphi^*$ is injective and that $(\varphi^*)^{-1}$ is continuous on \linebreak $\varphi^*(\HOM_w(\FR{B},M))$.  Showing that $\varphi^*$ is injective is a simple exercise and will be left to the reader.  To show the continuity of $(\varphi^*)^{-1}$ assume that $\varphi^*([\pi_n]) \rightarrow \varphi^*([\pi])$ in $\HOM_w(\FR{A},M)$.  This is the same as saying $[\pi_n \circ \varphi] \rightarrow [\pi \circ \varphi]$.  We will demonstrate that $[\pi_n] \rightarrow [\pi]$.  Let $\gamma_n \in [\pi_n \circ \varphi]$ be representatives such that $\gamma_n(a) \rightarrow \pi(\varphi(a))$ in $M$ under the trace norm.  Since $\gamma_n \in [\pi_n \circ \varphi]$ we get that $\ker(\varphi) \subseteq \ker(\gamma_n)$ for every $n$.  Thus by isomorphism theorems, we can write $\gamma_n = \delta_n \circ \varphi$ for some $*$-homomorphism $\delta_n: \FR{B} \rightarrow M$ for every $n$.  For $b \in \FR{B}$, we have that $b = \varphi(a)$ for some $a \in \FR{A}$; thus \[\delta_n(b) = \delta_n(\varphi(a)) = \gamma_n(a) \rightarrow \pi(\varphi(a)) = \pi(b).\]

It remains to show that $[\delta_n] = [\pi_n]$ for every $n$.  Fix $n$; there exists $\left\{u_k(n)\right\} \subset \mathcal{U}(M)$ such that for every $a \in \FR{A}$,
\[u_k(n) \gamma_n(a)u_k(n)^*  \rightarrow \pi_n(\varphi(a)).\]
Let $b \in \FR{B}$, and let $a \in \FR{A}$ be such that $\varphi(a) = b$.  Then 
\begin{align*}u_k(n)\delta_n(b)u_k(n)^* =& \quad u_k(n)\delta_n(\varphi(a))u_k(n)^*\\
 =& \quad u_k(n)\gamma_n(a)u_k(n)^*\\
 \rightarrow & \quad \pi_n(\varphi(a)) \\
=& \quad \pi_n(b).\qedhere
\end{align*}
\end{proof}

%Continuity

Similarly, a unital $*$-homomorphism $\psi: M_1 \rightarrow M_2$ between McDuff factors $M_1$ and $M_2$ also induces a map \[\psi_*: \HOM_w(\mathfrak{A},M_1) \rightarrow \HOM_w(\mathfrak{A}, M_2)\] given by $\psi_*([\pi]) = [\psi \circ \pi]$.

\begin{prop}\label{psitilde}
The induced map $\psi_*$ is well-defined, continuous, and affine.
\end{prop}

\begin{proof}
The fact that $\psi$ is a unital $*$-homomorphism guarantees that $\psi_*$ is well-defined.  Continuity is also routine.

To show that $\psi_*$ is affine, we will show that for $[\pi],[\rho] \in \HOM_w(\mathfrak{A},M_1)$ and for a projection $p \in \mathcal{P}(R)$ we have 
\begin{align*}
[\psi(\sigma_{M_1}(\pi \otimes p) + \sigma_{M_1}(\rho \otimes p^{\perp}))]  &=  [\sigma_{M_2}((\psi \circ \pi) \otimes p) + \sigma_{M_2}((\psi \circ \rho) \otimes p^{\perp})]\\
&\text{ or}\\
[\psi\circ \sigma_{M_1}(\pi \otimes p + \rho \otimes p^{\perp})]  & =  [\sigma_{M_2} \circ (\psi \otimes \text{id}_R)(\pi \otimes p + \rho \otimes p^{\perp})]
\end{align*} Here $\sigma_{M_1}$ and $\sigma_{M_2}$ are regular isomorphisms.

Thus it suffices to show that $\psi \circ \sigma_{M_1} \sim \sigma_{M_2}\circ(\psi \otimes \text{id}_R)$; or equivalently, $\sigma_{M_2}^{-1} \circ \psi \sim (\psi \otimes \text{id}_R)\circ \sigma_{M_1}^{-1}$.  Since $\sigma_{M_1}$ and $\sigma_{M_2}$ are regular isomorphisms, 
\begin{align*}
\sigma_{M_2}^{-1} \circ \psi & \sim  (\text{id}_{M_2} \otimes 1_R) \circ \psi\\
& =  \psi \otimes 1_R\\
& =  (\psi \otimes \text{id}_R) \circ (\text{id}_{M_1} \otimes 1_R)\\
& \sim  (\psi \otimes \text{id}_R) \circ \sigma_{M_2}^{-1}.\qedhere
\end{align*}
\end{proof}

\noindent  One can easily see that $\HOM_w(\FR{A},\cdot)$ satisfies the chain rule and preserves identities; so $\HOM_w(\FR{A},\cdot)$ is a covariant functor from the category of McDuff II$_1$ factors to the category of affine metric spaces.

\section{Connection to the Trace Space}\label{tracespace}

Given $[\pi] \in \HOM_w(\FR{A},M)$, we can assign to it a trace on $\FR{A}$ in the following natural way.

\begin{dfn}\label{tracecorr}
For a separable unital tracial $C^*$-algebra $\FR{A}$ and a McDuff factor $M$, let $\alpha_{(\FR{A},M)}: \HOM_w(\FR{A},M) \rightarrow T(\FR{A})$ be the map given by $\alpha_{(\FR{A},M)}([\pi]) = \tau_M \circ \pi$ where $\tau_M$ is the unique normalized trace of $M$, and $T(\FR{A})$ denotes the trace space of $\FR{A}$ (see \eqref{tracedef}).
\end{dfn}

\begin{prop}
For any separable unital tracial $C^*$-algebra $\FR{A}$ and McDuff factor $M$ the map $\alpha_{(\FR{A},M)}: \HOM_w(\FR{A},M) \rightarrow T(\FR{A})$ is well-defined, continuous (from the $d$-metric to the weak-$*$ topology), and affine.
\end{prop}

\begin{proof}
That $\alpha_{(\FR{A},M)}$ is well-defined and continuous follows from the continuity of $\tau_M$.  To see that $\alpha_{(\FR{A},M)}$ is affine,  let $[\pi],[\rho] \in \HOM_w(\FR{A},M)$ and $t \in [0,1]$.  Then for $a \in \FR{A}$ and $p \in \mathcal{P}(R)$ with $\tau(p) = t$, letting $\alpha$ stand for $\alpha_{(\FR{A},M)}$ we have
	\begin{align}
		\alpha(t[\pi] + (1-t)[\rho])(a) &= \alpha(\sigma_M(\pi \otimes p + \rho \otimes p^{\perp}))(a) \notag\\
		& =  \tau(\sigma_M(\pi(a) \otimes p)) + \tau(\sigma_M(\rho(a) \otimes p^{\perp})) \notag\\
		& =  t\tau(\pi(a)) + (1-t)\tau(\rho(a)) \label{uniquetrace} \\
		& =  (t\alpha([\pi]) + (1-t)\alpha([\rho]))(a).\notag
	\end{align}

Note that \eqref{uniquetrace} follows from the fact that $\tau(\sigma_M(\cdot \otimes p))$ is a trace on $M$ that evaluates to $t$ at $1_M$, and thus by the uniqueness of trace, we get that $\tau(\sigma_M(\cdot \otimes p)) = t\tau_M(\cdot)$.
\end{proof}

\subsection{Nuclear and Hyperfinite Cases}

In this subsection, we investigate the cases in which $\alpha_{(\FR{A},M)}$ is an affine homeomorphism for every McDuff $M$.  We first prove a technical lemma.

\begin{lem}\label{unitarysubseq}
Let $\pi,\pi_k: \FR{A}\rightarrow M$ be $*$-homomorphisms for $k \in \mathbb{N}$ and consider the $*$-homomorphisms 
\[\begin{array}{ccc}
(\pi_k)_\mathcal{U}: \FR{A} \rightarrow M^\mathcal{U}& \text{ and } & \pi^\mathcal{U}: \FR{A} \rightarrow M^\mathcal{U}
\end{array} \]
given by
\[\begin{array}{ccc}
(\pi_k)_\mathcal{U}(a) = (\pi_k(a))_\mathcal{U} & \text{ and } & \pi^\mathcal{U}(a) = (\pi(a))_\mathcal{U}\\
\end{array}\]
where $\mathcal{U}$ is a free ultrafilter on $\mathbb{N}$  (see Appendix A of \cite{brownozawa}).  If $(\pi_k)_\mathcal{U}$ is unitarily equivalent to $\pi^\mathcal{U}$, then there is a subsequence $\left\{k_j\right\}$ such that $[\pi_{k_j}] \rightarrow [\pi]$ in $\HOM_w(\FR{A},M)$.
\end{lem}

\begin{proof}

Let $u = (u_k)_\mathcal{U} \in \mathcal{U}(M^\mathcal{U})$ be such that for every $a \in \FR{A},$ \[ (\pi(a))_\mathcal{U} = (u_k\pi_{k}(a)u_k^*)_\mathcal{U}.\]  Let $\left\{a_i\right\} \subseteq \FR{A}_{\leq 1}$ be dense in the unit ball of $\FR{A}$.  Put \[A_j:= \bigcap_{1\leq i \leq j} \left\{k: ||u_k\pi_{k}(a_i)u_k^* - \pi(a_i)||_2< \frac{1}{j}\right\}.\]  We have $A_j \in \mathcal{U}$ for every $j$.  Pick $k_1 \in A_1$, and for $j>1$, pick $k_j \in A_j \cap \left\{k>k_{j-1}\right\}$.  We claim that $[\pi_{k_j}] \rightarrow [\pi]$ as $j \rightarrow \infty$.  Fix $\varepsilon > 0$.  Let $a \in \FR{A}_{\leq 1}$, and let $i \in \mathbb{N}$ be such that $\ds ||a-a_i||<\frac{\varepsilon}{4}$.  Let $J \in \mathbb{N}$ be such that $i \leq J$ and $\ds \frac{1}{J} < \frac{\varepsilon}{2}$.  Then for $j > J$, since $k_j \in A_j$, we have
\begin{align*}
||u_{k_j}\pi_{k_j}(a)u_{k_j}^* - \pi(a)||_2  \quad \leq& \quad ||u_{k_j}\pi_{k_j}(a)u_{k_j}^* - u_{k_j}\pi_{k_j}(a_i)u_{k_j}^*||_2 +\\
& \quad ||u_{k_j}\pi_{k_j}(a_i)u_{k_j}^*- \pi(a_i)||_2 + ||\pi(a_i) - \pi(a)||_2\\
  <&\quad \frac{\varepsilon}{4} + \frac{1}{j} + \frac{\varepsilon}{4}\\
< &\quad \varepsilon. 
\end{align*}
Since $\operatorname{Ad}(u_{k_j}) \circ \pi_{k_j} \in [\pi_{k_j}]$, this gives $[\pi_{k_j}] \rightarrow [\pi]$.
\end{proof}

\begin{thm}\label{nuctrace}
If $\FR{A}$ is nuclear, then $\alpha_{(\FR{A},M)}$ is an affine homeomorphism for any McDuff $M$.  In particular, $\HOM_w(\FR{A},M)$ is affinely homeomorphic to $T(\FR{A})$.
\end{thm}

\begin{proof}
\textbf{Injective:} When $\FR{A}$ is nuclear we have
	\begin{align}
		(\alpha([\pi]) = \alpha([\rho])) &\Leftrightarrow  (\tau \circ \pi = \tau \circ \rho) \notag\\
		%& \Leftrightarrow & M-\text{rank} \circ \pi = M-\text{rank}\circ \rho \label{lem3}\\
		& \Leftrightarrow  ([\pi] = [\rho]) \label{thm5}
	\end{align}

\noindent Here \eqref{thm5} follows from Lemma 3 and Theorem 5 of \cite{ding-hadwin}.  

\noindent \textbf{Surjective:}  It is well known that in the nuclear case, every trace gives a hyperfinite GNS closure.  So every trace lifts through $R\subset M$.  Thus $\alpha_{(\FR{A},M)}$ is surjective.

\noindent \textbf{Bicontinuous:}  Let $T_n \rightarrow T$ in $T(\FR{A})$. Let $[\pi_n] = \alpha_{(\FR{A},M)}^{-1} (T_n)$ and $[\pi] = \alpha_{(\FR{A},M)}^{-1}(T)$.  We must show that $[\pi_n] \rightarrow [\pi]$.  We will show this by appealing to the following standard topological fact: for a sequence $\left\{a_n\right\}$, if every subsequence $\left\{a_{n(k)}\right\}$ has a sub-subsequence $\left\{a_{n(k_j)}\right\}$ converging to $a$, then $\left\{a_n\right\}$ converges to $a$. Let $\left\{[\pi_{n(k)}]\right\}$ be a subsequence of $\left\{[\pi_n]\right\}$. Now consider the homomorphism $(\pi_{n(k)})_\mathcal{U}: \FR{A} \rightarrow M^\mathcal{U}$ where $\mathcal{U}$ is a free ultrafilter on $\mathbb{N}$.  By the convergence of the induced traces, we get that \[\tau_{M^\mathcal{U}} \circ (\pi_{n(k)})_\mathcal{U} = \tau_{M^\mathcal{U}} \circ \pi^\mathcal{U}.\]  And by Corollary 3.4 of \cite{autoultra}, we get that $(\pi_{n(k)})_\mathcal{U}$ is unitarily equivalent to $\pi^\mathcal{U}$.  So by Lemma \ref{unitarysubseq}, there is a sub-subsequence $\left\{n(k_j)\right\}$ such that $[\pi_{n(k_j)}]\rightarrow [\pi]$.
\end{proof}

%\begin{rmk}
%This isomorphism might extend to TAF-algebras.  
%\end{rmk}

\begin{exmpl}\label{blackadar}
Theorem 3.10 of \cite{blackadar} says that for any metrizable Choquet simplex $\Delta$, there exists a simple unital AF algebra $\FR{B}$ such that $T(\FR{B})$ is affinely homeomorphic to $\Delta.$  Combining this fantastic result with Theorem \ref{nuctrace} tells us that any (separable) metrizable Choquet simplex $\Delta$ can arise as $\HOM_w(\FR{B},M)$ for some (separable) $\FR{B}$.
\end{exmpl}

We now work to characterize the algebras $\FR{A}$ for which $\alpha_{(\FR{A},M)}$ is an affine homeomorphism for every McDuff $M$.  

\begin{dfn}
Let \[T(\FR{A},M) := \left\{T \in T(\FR{A}) : \text{ there exists } \pi: \FR{A} \rightarrow M \text{ such that }  T = \tau_{M} \circ \pi\right\}.\]  For $T \in T(\FR{A},M)$, we say \textquotedblleft $T$ lifts through $M$.\textquotedblright
\end{dfn}

\begin{thm}\label{injR}
For any $\FR{A}$, the map $\alpha_{(\FR{A},R)}: \HOM_w(\FR{A},R) \rightarrow T(\FR{A})$ is always a homeomorphism onto its image.  In particular, $\HOM_w(\FR{A},R) \cong T(\FR{A},R)$ (affine homeomorphism). 
\end{thm}

\begin{proof}
Let $\pi,\rho: \FR{A}\rightarrow R$ be unital $*$-homomorphisms such that $\alpha_{(\FR{A},R)}([\pi]) = \alpha_{(\FR{A},R)}([\rho])$.  We will show that $[\pi]=[\rho]$.  

Consider the following map \[\varphi: W^*(\pi(\FR{A})) \rightarrow W^*(\rho(\FR{A}))\] densely defined by \[\varphi(\pi(a)) = \rho(a).\]  The assumption that $\tau \circ \pi = \tau \circ \rho$ gives that all the $*$-moments in $\pi(\FR{A})$ agree with those of $\rho(\FR{A})$.  So this is a well-defined $*$-homomorphism.  Now $W^*(\pi(\FR{A}))\subseteq R$ is hyperfinite.  So let $\ds \left\{B_n\right\}_{n=1}^{\infty}$ be an increasing sequence of finite dimensional algebras such that $\ds W^*(\pi(\FR{A})) = W^*(\cup_{n=1}^{\infty} B_n)$.  Now define \[\varphi_n := \varphi|_{B_n},\] and let \[\iota_n: B_n\rightarrow B_n\] be the identity on $B_n$.  Observe that $\tau \circ \varphi_n = \tau \circ \iota_n$ on $B_n$.  So by Theorem \ref{nuctrace} we have that 
\begin{equation}\label{fidiequiv}
\varphi_n \sim \iota_n.
\end{equation}
(In fact, $\varphi_n$ and $\iota_n$ are unitarily equivalent; but we will not need this.)

We will now use the equivalent formulation of weak approximate unitary equivalence as given in the introduction to complete the proof.  Let $a_1, \dots, a_k \in$\linebreak  $W^*(\pi(\FR{A}))_{\leq 1}$ and $\varepsilon > 0$ be given.  By Kaplansky density, there is a $K$ such that there are $b_1, \dots, b_k \in (B_K)_{\leq 1}$  with $||a_j - b_j||_2 < \frac{\varepsilon}{3}$ for all $1\leq j \leq k$.  And by equality of $*$-moments, we have as a consequence that $||\varphi(a_j) - \varphi(b_j)||_2 \leq \frac{\varepsilon}{3}$.  Now by \eqref{fidiequiv} there is a unitary $u \in\mathcal{U}(R)$ such that \[||\iota_K(b_j) - u \varphi_K(b_j)u^*||_2 = ||b_j - u\varphi_K(b_j)u^*||_2 < \frac{\varepsilon}{3} \quad \forall \quad  1 \leq j \leq k.\]  Thus we have for every $1\leq j \leq k,$
\begin{align*}
||a_j - u \varphi(a_j)u^*||_2   \quad\leq&\quad   ||a_j - b_j||_2 + ||b_j - u\varphi_K(b_j)u^*||_2 \;+\\
& \quad  ||u\varphi_K(b_j)u^* - u \varphi(a_j)u^*||_2\\
 =&  \quad ||a_j - b_j||_2 + ||b_j - u\varphi_K(b_j)u^*||_2 + ||\varphi(b_j) -  \varphi(a_j)||_2\\
 < & \quad \frac{\varepsilon}{3} + \frac{\varepsilon}{3} + \frac{\varepsilon}{3}\\
 =&  \quad \varepsilon.
\end{align*}

So $\text{id}_{W^*(\pi(\FR{A}))} \sim \varphi$. Thus $[\pi] = [\rho]$, and $\alpha_{(\FR{A},R)}$ is injective.  

It remains to show that $\alpha_{(\FR{A},R)}^{-1}$ is continuous on $T(\FR{A},R)$. We proceed similarly to the bicontinuous part of the proof of Theorem \ref{nuctrace}.  Let $T_n \rightarrow T$ in $T(\FR{A},R)$ in the weak-$*$ sense.  And let $\pi_n,\pi: \FR{A} \rightarrow R$ be such that $\tau \circ \pi_n = T_n$ and $\tau \circ \pi = T$.  Let $\left\{n(k)\right\}$ be a subsequence.  Consider $(\pi_{n(k)})_\mathcal{U}$ and $\pi^\mathcal{U}$ for a free ultrafilter $\mathcal{U}$ of $\mathbb{N}$.  By assumption, $\tau_{R^\mathcal{U}} \circ (\pi_{n(k)})_\mathcal{U} = T$.  So by the uniqueness of the GNS construction, $W^*((\pi_{n(k)})_\mathcal{U}(\FR{A})) \cong W^*(\pi^\mathcal{U}(\FR{A}))$ is hyperfinite.  Let $R_1,R_2 \subset R^\mathcal{U}$ be copies of $R$ such that 
\[(\pi_{n(k)})_\mathcal{U}: \FR{A} \rightarrow R_1 \subset R^\mathcal{U}\]
and
\[\pi^\mathcal{U}: \FR{A} \rightarrow R_2\subset R^\mathcal{U}.\]
We can assume without loss of generality that $R_1=R_2=R$ (see Lemma 2.9 of \cite{jung}).  

So we have two $*$-homomorphisms
\[(\pi_{n(k)})_\mathcal{U}: \FR{A} \rightarrow R\]
and
\[\pi^\mathcal{U}: \FR{A} \rightarrow R\]
such that $\tau_R \circ (\pi_{n(k)})_\mathcal{U} = \tau_R \circ \pi^\mathcal{U}$, or $\alpha_{(\FR{A},R)}((\pi_{n(k)})_\mathcal{U}) = \alpha_{(\FR{A},R)}(\pi^\mathcal{U})$.  Since we have from above that $\alpha_{(\FR{A},R)}$ injective, we can conclude $(\pi_{n(k)})_\mathcal{U}$ is weakly approximately unitarily equivalent to $\pi^\mathcal{U}$ in $R \subset R^\mathcal{U}$.  And by Theorem 3.1 of \cite{autoultra} we have that $(\pi_{n(k)})_\mathcal{U}$ is unitarily equivalent to $\pi^\mathcal{U}$.  Then Lemma \ref{unitarysubseq} tells us that there is a sub-subsequence $\left\{n(k_j)\right\}$ such that $[\pi_{n(k_j)}]\rightarrow [\pi]$.  So we have shown that for any subsequence $\left\{[\pi_{n(k)}]\right\}$ of $\left\{[\pi_n]\right\}$ there is a sub-subsequence $\left\{[\pi_{n(k_j)}]\right\}\subset \left\{[\pi_{n(k)}]\right\}$ such that $[\pi_{n(k_j)}] \rightarrow [\pi]$.  Thus \[\alpha_{(\FR{A},R)}^{-1}(T_n) = [\pi_n] \rightarrow [\pi] = \alpha_{(\FR{A},R)}^{-1}(T).\qedhere\]
\end{proof}

\begin{dfn}[Definition 3.2.1, \cite{invar}]
A trace $T \in T(\FR{A})$ is called \emph{uniform amenable} if there exists a sequence of unital completely positive maps $\varphi_n: \FR{A} \rightarrow \mathbb{M}_{k(n)}$ such that 
\[\lim_n ||\varphi_n(ab) - \varphi_n(a)\varphi_n(b)||_2 =0\] for all $a,b \in \FR{A}$, and
\[\lim_n ||T - \text{tr}_{k(n)} \circ \varphi_n||_{\FR{A}^*}=0\] 
where $||\cdot||_{\FR{A}^*}$ is the natural norm on the dual of $\FR{A}$.  Let $\text{UAT}(\FR{A})$ denote the set of all such traces.
\end{dfn}

From Theorem \ref{injR} we have that $\HOM_w(\FR{A},R) \cong T(\FR{A},R)$ as convex sets.  And by (5) of Theorem 3.2.2 of \cite{invar} it follows that $T(\FR{A},R) = \text{UAT}(\FR{A})$. We can now give our characterization theorem.

\begin{thm}\label{affhomchar} 
The following are equivalent:
\begin{enumerate}

	\item $\alpha_{(\FR{A},M)}$ is an affine homeomorphism for every McDuff $M$;

	\item $\alpha_{(\FR{A},R)}$ is an affine homeomorphism;

	\item $\alpha_{(\FR{A},R)}$ is surjective;

	\item UAT$(\FR{A}) = T(\FR{A},R) = T(\FR{A})$.

\end{enumerate}

\end{thm}

\begin{proof}
The implications $((1)\Rightarrow (2))$ and $((2)\Rightarrow(3))$ are obviously true.

The observation that $\alpha_{(\FR{A},R)}(\HOM_w(\FR{A},R)) = T(\FR{A},R) = \text{UAT}(\FR{A})$ shows the equivalence $((3)\Leftrightarrow(4))$.

It remains to show $((4)\Rightarrow (1))$: If $M$ is such that $\alpha_{(\FR{A},M)}$ is not injective, then there are homomorphisms \[\pi,\rho: \FR{A}\rightarrow M\] such that $\tau_M\circ \pi = \tau_M\circ \rho$ but $[\pi]\neq [\rho]$.  Then by Theorem \ref{injR} we have that $\tau\circ \pi \notin T(\FR{A},R)$ -- a contradiction.  So $\alpha_{(\FR{A},M)}$ must be injective for every $M$.  If $M$ is such that $\alpha_{(\FR{A},M)}$ is not surjective then there exists $T \in T(\FR{A})$ such that $T$ does not lift through $M$; thus $T$ cannot lift through $R$ either.  So again $T(\FR{A},R) \neq T(\FR{A})$ -- a contradiction.  So $\alpha_{(\FR{A},M)}$ must be bijective for every $M$.  To show that $\alpha_{(\FR{A},M)}^{-1}$ is continuous for every $M$, we use the assumption that $T(\FR{A}) = T(\FR{A},R)$ along with an argument identical to the justification of the continuity of $\alpha_{(\FR{A},R)}^{-1}$ in the proof of Theorem \ref{injR}.
\end{proof}

\noindent Thus, the class of algebras $\FR{A}$ for which $\alpha_{(\FR{A},M)}$ is an affine homeomorphism for all McDuff $M$ is exactly the class of all $\FR{A}$ such that for any trace $T \in T(\FR{A})$, the GNS representation of $\FR{A}$ induced by $T$ has a hyperfinite von Neumann closure. 

\begin{exmpl}
The above class of algebras is strictly larger than the class of nuclear algebras.  Dadarlat's example of a non-nuclear subalgebra of an AF-algebra in \cite{dadarlat} is an example of a non-nuclear algebra whose tracial GNS representations are hyperfinite.
\end{exmpl}

%We also observe a similar characterization where we fix the second argument and let the first vary. The following theorem is a consequence of Theorem 6.3 of \cite{ciuperca} combined with Lemma 3 of \cite{ding-hadwin}.
%
%\begin{thm}\cite{ding-hadwin}\cite{ciuperca}
%A II$_1$ factor $M$ is hyperfinite if and only if $\alpha_{(\FR{A},M\otimes M)}$ is injective for every $\FR{A}$.
%\end{thm}

\subsection{Examples}\label{examples}

By Theorem \ref{nuctrace} we know that $\alpha_{(\FR{A},M)}$ can be both injective and surjective.  The following examples will demonstrate that the other three cases where one or both properties fail are possible.  This suggests that $\HOM_w(\FR{A},M)$ is a rich object with deep and interesting subtleties.

\begin{exmpl}\label{injfail}
\textbf{\tql Forgetful Trace.\tqr}This example shows that $\alpha_{(\FR{A},M)}$ is not always injective -- confirming that $\HOM_w(\FR{A},M)$ carries information different from that of $T(\FR{A})$.  The strategy of this example also reveals the usefulness of non-approximately inner automorphisms of McDuff factors.  Let $M$ be a McDuff factor with a non-approximately inner automorphism $\varphi$ (e.g. $\ds M = \otimes_{\mathbb{Z}} L(\mathbb{F}_2)$ satisfies this property as mentioned in Remark \ref{sakairmk}).  Let $\FR{A}$ be a separable, $||\cdot||_2$-dense $C^*$-subalgebra of $M$.  Let $\pi: \FR{A}\rightarrow M$ be the identity inclusion, and let $\rho = \varphi \circ \pi$.  Then we claim \[\tau_M \circ \rho = \tau_M \circ \pi,\] but \[[\pi] \neq [\rho].\]  Since $\tau_M \circ \varphi$ is also a trace on the II$_1$ factor $M$, we have $\tau_M = \tau_M \circ \varphi$.  Then $\tau_M(\pi(a)) = \tau_M(\varphi(\pi(a))) = \tau_M(\rho(a))$.  Now suppose for the sake of contradiction that $[\pi] = [\rho]$.  Then there is a sequence of unitaries $\left\{u_n\right\}$ in $M$ such that for every $a \in \FR{A}$ we have \[\lim_n ||\varphi(\pi(a)) - u_n\pi(a)u_n^*||_2  = \lim_n ||\varphi(a) - u_nau_n^*||_2= 0.\]  Then the $||\cdot||_2$-density of $\pi(\FR{A})=\FR{A} \subset M$ implies that for every $x \in M$ we have \[\lim_n||\varphi(x) - u_n x u_n^*||_2 = 0\] meaning that $\varphi$ is approximately inner -- a contradiction.

If we further insist that $\FR{A}$ has a unique trace (by throwing in enough unitaries via Dixmier approximation -- see Lemma \ref{sepsubfac}), then this is an example of $\alpha_{(\FR{A},M)}$ failing to be injective while remaining surjective.
\end{exmpl}

\begin{exmpl}\label{surjfail}
By Proposition 3.5.1 of \cite{invar} we have that $T(\FR{A},R)$ is a weakly closed subset of $T(\FR{A})$.  It is not true in general, however, that $T(\FR{A},R)$ is weak-$*$ closed in $T(\FR{A})$.  By Remark 4.1.7 of \cite{invar} if $\Gamma$ is a non-amenable, residually finite, discrete group (e.g. $\mathbb{F}_n$) then the $T(C^*(\Gamma),R)$ is not weak-$*$ closed. So for $\FR{A} = C^*(\Gamma)$ where $\Gamma$ is a non-amenable, residually finite, discrete group, we have that $\alpha_{(\FR{A},R)}$ fails to be surjective while remaining injective. Furthermore, $\HOM_w(\FR{A},R)$ is not compact: if it were, then by continuity $\alpha_{(\FR{A},R)}(\HOM_w(\FR{A},R))=T(\FR{A},R) = \text{UAT}(\FR{A})$ would also be weak-$*$ compact and thus weak-$*$ closed -- a contradiction.
\end{exmpl}

We will need the following lemma for the next example.

\begin{lem}\label{sepsubfac}
If $Y$ is a separable von Neumann subalgebra of a II$_1$ factor $X$ with $1_Y = 1_X$, then there is a separable II$_1$ factor $M$ contained unitally in $X$ that contains $Y$: $Y \subset M \subset X$.
\end{lem}

\begin{proof}
This proof will heavily rely on Dixmier's approximation theorem:  For $N$ a finite von Neumann algebra and $x \in N$ then we have \[\overline{\text{conv}}\left\{uxu^*|u \in \mathcal{U}(N)\right\} \cap \mathcal{Z}(N) = \left\{T(x)\right\}\] where $T$ is the unique center-valued trace, and the closure is in the norm topology.

We will recursively construct an increasing sequence $\left\{Y_i\right\}_{i=0}^{\infty}$ of separable subalgebras of $X$ and claim that $M = W^*(\cup_{i=0}^{\infty} Y_i)$ is the desired factor.  Let $Y_0 = Y$.   We will assume that $Y_i$ has been constructed and go about constructing $Y_{i+1}$.  Let $\left\{y(i,j)\right\}_{j=1}^{\infty} \subset (Y_i)_{\leq 1}$ be weakly dense in $(Y_i)_{\leq 1}$.  By Dixmier, for any $j$ and for any $k \in \mathbb{N}$ there are unitaries $u_1(i,j,k),\dots, u_{n(i,j,k)}(i,j,k) \in \mathcal{U}(X)$ and scalars $\alpha_1(i,j,k),\dots, \alpha_{n(i,j,k)}(i,j,k) \in (0,1)$ with $\sum_{p=1}^{n(i,j,k)} \alpha_p(i,j,k) = 1$ such that \[\Big|\Big|\sum_{p=1}^{n(i,j,k)} \alpha_p(i,j,k)u_p(i,j,k)y(i,j)u_p(i,j,k)^* - \tau(y(i,j))I\Big|\Big|< \frac{1}{k}.\]  Then we let \[Y_{i+1} = W^*(Y_i \cup (\cup_{j=1}^{\infty} \cup_{k=1}^{\infty} \left\{u_1(i,j,k),\dots,u_{n(i,j,k)}(i,j,k)\right\})).\]

Let $M := W^*(\cup_{i=0}^{\infty} Y_i)$.  To show that $M$ is a factor we will show that it has a unique unital trace (given by restriction of the unital trace $\tau$ on $X$).  Let $T$ be a unital trace on $M$ and let $m \in M_{\leq 1}$. It will suffice to show that for any $\varepsilon >0, |T(m) - \tau(m)| < \varepsilon$.  Fix $\varepsilon > 0$.  Let $K \in \mathbb{N}$ be such that $\frac{1}{K} < \varepsilon$.  And let $i(m),j(m)$ be such that $y(i(m),j(m)) \in Y_{i(m)}$ with $|T(m) - T(y(i(m),j(m))|<\frac{\varepsilon}{3}$ and $|\tau(m) - \tau(i(m),j(m))| < \frac{\varepsilon}{3}$ (guaranteed by the weak continuity of both traces).

For brevity let 
\begin{align*}
y &= y(i(m),j(m)),& u_p &= u_p(i(m),j(m),3K),\\
\alpha_p &= \alpha_p(i(m),j(m),3K),& n &= n(i(m),j(m),3K)
\end{align*}
and consider

\begin{align*}
|T(y) - \tau(y)| & =   \Big|T\Big(\sum_{p=1}^n \alpha_p u_p y u_p^*\Big) - \tau(y)\Big|\\
& =  \Big|T\Big(\sum_{p=1}^n \alpha_p u_p y u_p^* - \tau(y)I\Big)\Big|\\
& \leq  \Big|\Big|\sum_{p=1}^n \alpha_p u_p y u_p^* - \tau(y)I\Big|\Big|\\
& <  \frac{1}{3K}\\
& < \frac{\varepsilon}{3}.
\end{align*}

Thus we have 

\begin{align*}
|T(m) - \tau(m)| & \leq  |T(m) - T(y)| + |T(y) - \tau(y)| + |\tau(y) - \tau(m)|\\
& <  \frac{\varepsilon}{3} + \frac{\varepsilon}{3} + \frac{\varepsilon}{3}\\
& =  \varepsilon,
\end{align*}

and we are done.
\end{proof}

\begin{dfn}
A map $\pi: (N,\sigma) \rightarrow (M,\tau)$ between two finite tracial von Neumann algebras is an \emph{embedding} if it is a unital, normal, injective, trace preserving $*$-homomorphism. We will write $N$ for $(N,\sigma)$ when no confusion can occur. We say that a tracial von Neumann algebra $N$ is \emph{embeddable} if there exists a trace-preserving unital embedding $\pi: N \rightarrow R^\mathcal{U}$.
\end{dfn}

\begin{exmpl}\label{simul}
\textbf{\tql Too Many Traces.\tqr}This example will provide an algebra $\FR{A}$ such that $\alpha_{(\FR{A},M)}$ simultaneously fails injectivity and surjectivity for some $M$.  This idea was suggested by N. Brown.  Let $N$ be an embeddable separable non-hyperfinite II$_1$-factor.  Let $\pi,\rho: N \rightarrow R^{\mathcal{U}}$ be two embeddings that are not unitarily equivalent (this is guaranteed by Lemma 2.9 of \cite{jung}). Let $Y := W^*(\pi(N) \cup \rho(N))$, and let $X = R^\mathcal{U}$.  We have that the separable algebra $Y$ is contained in the (nonseparable) II$_1$ factor $X$; so by Lemma \ref{sepsubfac} there is a separable II$_1$ factor $M$ such that $Y \subset M \subset X$. We claim that $\pi$ and $\rho$ are not weakly approximately unitarily equivalent in $M$.  If they were weakly approximately unitarily equivalent in $M$ then they will also be weakly approximately unitarily equivalent in $R^{\mathcal{U}}$; but then by Theorem 3.1 of \cite{autoultra} we have that $\pi$ and $\rho$ are unitarily equivalent in $R^{\mathcal{U}}$ -- a contradiction.  Consider $\pi \otimes 1_R, \rho \otimes 1_R:  N \rightarrow M\otimes R$.  By Theorem \ref{presineq}, since $\pi$ and $\rho$ are inequivalent, we have that $\pi \otimes 1_R$ is not weakly approximately unitarily equivalent to $\rho \otimes 1_R$ (be assured that the proof of Theorem \ref{presineq} does not depend on this example).

Now let $\FR{A} = C^*(\mathbb{F}_{\infty})$.  And let $\zeta: C^*(\mathbb{F}_{\infty}) \rightarrow N$ be a $*$-monomorphism with weakly dense image as guaranteed by Proposition 3.1 of \cite{con}.  Let $\hat{\pi}$ and $\hat{\rho}$ be given by \[\hat{\pi} = (\pi\otimes 1_R) \circ \zeta: C^*(\mathbb{F}_{\infty}) \rightarrow N \rightarrow M \otimes R\] and \[ \hat{\rho} = (\rho \otimes 1_R) \circ \zeta: C^*(\mathbb{F}_{\infty}) \rightarrow N \rightarrow M \otimes R.\]  Then we clearly have that $[\hat{\pi}] \neq [\hat{\rho}]$ but $\alpha_{(\FR{A},M\otimes R)}([\hat{\pi}]) = \alpha_{(\FR{A},M\otimes R)}([\hat{\rho}])$.

Another consequence of Proposition 3.1 of \cite{con} is that $\FR{A} = C^*(\mathbb{F}_{\infty})$ enjoys the property that for any McDuff factor $S$ there is a trace $T_S \in T(\FR{A})$ such that $\pi_{T_S}(\FR{A})'' \cong S$ (where $\pi_{T_S}$ is the GNS representation corresponding with $T_S$).  By \cite{Ozawa}, there is no separable universal II$_1$ factor, and so we can conclude that there is no separable universal McDuff factor.  So let $S$ be such that $S$ does not embed into $M \otimes R$.  Then we have that $T_S \in T(\FR{A})$ as described above does not lift through $M \otimes R$.  Thus $T_S \notin \alpha_{(\FR{A},M\otimes R)}(\HOM_w(\FR{A},M\otimes R))$.  So we have that $\alpha_{(\FR{A},M\otimes R)}$ is neither injective nor surjective.
\end{exmpl}

\subsection{An Alternative Characterization of Hyperfiniteness}

Investigating the connection between weak approximate unitary equivalence and preservation of a given trace has led us to a characterization of a finite tracial ($R^\mathcal{U}$-embeddable) hyperfinite von Neumann algebra that we believe to be new.

A result of Jung gives a characterization of hyperfiniteness using embeddings into $R^\mathcal{U}$.  We state it as follows.

\begin{thm}[Lemma 2.9, \cite{jung}]\label{jung}
A separable tracial embeddable von Neumann algebra $N$ is hyperfinite if and only if any two embeddings $\pi,\rho: N \rightarrow R^\mathcal{U}$ are conjugate by a unitary in $R^\mathcal{U}$. 
\end{thm}

The characterization we present in the following theorem frames Jung's result in terms of embeddings into separable algebras -- removing (most of) the ultrapower language from the characterization.  This characterization may be known to experts; but we have not seen it in the literature.

\begin{thm}\label{sepchar}
Let $N$ be a separable tracial embeddable von Neumann algebra.  Then $N$ is hyperfinite if and only if for every separable McDuff II$_1$-factor $M$, any two embeddings $\pi,\rho: N \rightarrow M$ are weakly approximately unitarily equivalent.
\end{thm}

\begin{proof}
\noindent \textbf{($\Rightarrow$):}  This implication is well-known, but we include the reasoning here for the sake of completeness. Assume that $N$ is hyperfinite.  Approximate $N$ with an increasing sequence of finite-dimensional subalgebras.  The conclusion follows from the fact that any two embeddings of a finite dimensional algebra into a II$_1$-factor are unitarily equivalent.

\noindent \textbf{($\Leftarrow$):}  The argument here is similar to the one found in Example \ref{simul}.  Assume that $N$ is not hyperfinite.  Then Jung's result says that there exist two embeddings $\pi, \rho: N \rightarrow R^\mathcal{U}$ such that $\pi$ and $\rho$ are not unitarily conjugate.  Let $Y := W^*(\pi(N)\cup \rho(N))$.  Then by Lemma \ref{sepsubfac} there is a separable II$_1$-factor $M_0 \subset R^\mathcal{U}$ containing $Y$.  Just as in Example \ref{simul}, we may argue that $\pi$ and $\rho$ are not weakly approximately unitarily equivalent in $M_0$.  And by Theorem \ref{presineq} (whose proof does not rely on this result), we have that $\pi\otimes 1_R: N \rightarrow M_0 \otimes R$ is not weakly approximately unitarily equivalent to $\rho \otimes 1_R: N \rightarrow M_0 \otimes R$.  So putting $M := M_0 \otimes R$, we are done.
\end{proof}

\begin{rmk}
Jung's approach to the characterization in \cite{jung} hinges on the concept of tubularity: a condition on neighborhoods of unitary orbits of the microstate spaces for the generators of the algebra.  We remark here that this concept of tubularity was preceded a decade earlier in \cite{hadwin} by Hadwin's concept of dimension ratio.  The dimension ratio is a quantity associated to the self adjoint generators of a tracial $C^*$-algebra. This dimension ratio quantifies tubularity in the sense that the dimension ratio of the generators is 0 if and only if the generators are tubular.  See \cite{jung} and \cite{hadwin} for the relevant definitions and theorems.
\end{rmk}

\section{Convex Analysis}\label{convexanalysis}

In the first part of this section we discuss a necessary condition for extreme points of $\HOM_w(\FR{A},M)$. A complete characterization of extreme points remains open, and thus the question of existence of extreme points in $\HOM_w(\FR{A},M)$ is also open.  In \S \ref{ultrasit} we provide characterizations for extreme points in broad cases.  The second part of this section discusses a natural relationship between quotients of $\FR{A}$ and faces of $\HOM_w(\FR{A},M)$. The discussion there provides a sufficient condition for the existence of extreme points. 

%While we provide answers in \S\ref{ultrasit} for the cases where either $\FR{A}$ or $M$ is amenable; the problem remains open in the general case.

%A nice characterization of extreme points (the relative commutant of the image is a factor) is demonstrated in Proposition 5.2 of\cite{topdyn}.  In our framework we do not have access to such a characterization because commutants are not preserved under weak approximate unitary equivalence.  For example, let $\FR{A}$ be a separable dense $C^*$-subalgebra of $R$ and consider \[[\text{id}_{\FR{A}}] \in \HOM_w(\FR{A}, R).\]  For the representative $\text{id}_{\FR{A}}$ we have $\text{id}_{\FR{A}}(\FR{A})'\cap R = \mathbb{C}$; but if we choose the representative $\sigma_R(\text{id}_{\FR{A}} \otimes 1_R)$ instead, then $\sigma_R(\text{id}_{\FR{A}}(\FR{A}) \otimes 1_R)' \cap R \cong R$.	By Theorem 3.1 of \cite{autoultra}, shifting to ultrapower target algebras removes this obstruction--see \S \ref{ultrasit}

\subsection{Factorial Closure}

We now proceed to establish a necessary (but not sufficient: see Example \ref{sufffail}) condition for extreme points in $\HOM_w(\FR{A},M)$.  Although relative commutants are not well-defined under weak approximate unitary equivalence in general, it is easy to see that the von Neumann closure of the image of a $*$-homomorphism is well-defined under weak approximate unitary equivalence up to $*$-isomorphism.  We state the necessary condition for extreme points as follows.

%\begin{prop}
%For $[\pi] \in \HOM_w(\FR{A},M)$, the correspondence $[\pi] \mapsto W^*(\pi(\FR{A}))$ is well-defined up to isometric $*$-isomorphism.
%\end{prop}
%
%\begin{proof}
%
%Let $[\pi_1] = [\pi_2]$.  Then there exists a sequence $\left\{u_n\right\} \subset \mathcal{U}(M)$ such that \[\lim_{n\rightarrow \infty}||\pi_2(a) - u_n\pi_1(a)u_n^*||_2 = 0\] for every $a \in \FR{A}$.  
%
%Define \[\varphi: \pi_1(\FR{A}) \rightarrow \pi_2(\FR{A})\] densely by \[\varphi(\pi_1(a)) = \pi_2(a) = \lim_{n\rightarrow \infty} u_n \pi_1(a) u_n^*.\]  We have that $\varphi$ is a $*$-homomorphism.  Also, \[||\varphi(\pi_1(a))||_2 = ||\pi_2(a)||_2 =  \lim_n||u_n \pi_1(a)u_n^*||_2 = ||\pi_1(a)||_2\] so $\varphi$ extends to $W^*(\pi_1(\FR{A}))$ in the following way: if $\ds A = \lim_i \pi_1(a_i)$ then $\varphi(A):= \lim_i \pi_2(a_i)$.  Let $\hat{\varphi}: W^*(\pi_1(\FR{A})) \rightarrow W^*(\pi_2(\FR{A}))$ denote this extension.  Since $\varphi$ is isometric, we get that $\hat{\varphi}$ is a well-defined isometric $*$-homomorphism.  Similarly construct $\hat{\psi}: W^*(\pi_2(\FR{A})) \rightarrow W^*(\pi_1(\FR{A})).$  Then it is clear to see that $\hat{\varphi}$ and $\hat{\psi}$ are mutual inverses. \end{proof}
%
%\noindent With the above proposition, the statement of the following theorem now makes sense. 

\begin{thm}\label{extfact}
If $[\pi] \in \HOM_w(\FR{A},M)$ is extreme, then $W^*(\pi(\FR{A}))$ is a factor.
\end{thm}

\begin{proof}
Using the isomorphism $\HOM_w(\FR{A},M) \cong \HOM_w(\FR{A},M\otimes R)$ we will show that if $[\pi] \in \HOM_w(\FR{A},M\otimes R)$ is extreme then $W^*(\pi(\FR{A}))$ is a factor.  We will argue by contrapositive and assume that $W^*(\pi(\FR{A}))$ is not a factor.  Then there exists a nontrivial central projection $z \in W^*(\pi(\FR{A}))$.  In particular $0<\tau(z)<1$.  We will now construct inequivalent $*$-homomorphisms $\rho_1$ and $\rho_2$ using $z$ and $z^{\perp}$ in the following way.  Let $p \in R$ be a projection such that $\tau(p) = \tau(z)$.  Let $\sigma_{M\otimes R} : M\otimes R\otimes R \rightarrow M \otimes R, \sigma_M: M \otimes R \rightarrow M$, and $\epsilon: R\otimes R$  be regular isomorphisms.  Let $v,w \in M \otimes R\otimes R$ be partial isometries with 
\begin{align*}
v^*v &= \sigma_{M\otimes R}^{-1}(z), & vv^* &= 1_M \otimes 1_R \otimes p, \\
 w^*w &= \sigma_{M\otimes R}^{-1}(z^{\perp}),&  ww^* &= 1_M \otimes 1_R \otimes p^{\perp}.
\end{align*}

We have $v+w \in \mathcal{U}(M\otimes R \otimes R)$.  Let $Ad(u)(x) = uxu^*$.  For $q \in \left\{p,p^{\perp}\right\}$, let $T_q : qRq \rightarrow R$ be an isomorphism.  We are now ready to define $\rho_1,\rho_2: \FR{A} \rightarrow M \otimes R$ by the following formulas.

\begin{align*}
\rho_1(a) &= (\text{id}_M \otimes \epsilon)\circ(\text{id}_M\otimes\text{id}_R \otimes T_p)\circ Ad(v+w)\circ\sigma_{M\otimes R}^{-1}(z\pi(a))\\
\rho_2(a) & =   (\text{id}_M \otimes \epsilon)\circ(\text{id}_M\otimes\text{id}_R \otimes T_p^{\perp})\circ Ad(v+w)\circ\sigma_{M\otimes R}^{-1}(z^{\perp}\pi(a)).
\end{align*}
By construction, we have $[\rho_1],[\rho_2] \in \HOM_w(\FR{A},M\otimes R)$.  

We claim that we will be done if we show the following two statements are true:
\begin{enumerate}
	\item If $t = \tau(z)$ then $t[\rho_1] + (1-t)[\rho_2] = [\pi]$.
	\item $[\rho_1] \neq [\pi]$.
\end{enumerate}
Indeed, if these two statements hold, then $[\pi]$ is not an extreme point.

\noindent (1):  By definition, we have that \[t[\rho_1] + (1-t)[\rho_2] = [\sigma_{M\otimes R} (\rho_1 \otimes p + \rho_2 \otimes p^{\perp})].\]  
And we have 
\[\rho_1(a)\otimes p = \big((\text{id}_{M\otimes R} \otimes p)\circ(\text{id}_M \otimes \epsilon) \circ (\text{id}_M \otimes \text{id}_R \otimes T_p)\big) \circ Ad(v+w) \circ \sigma_{M\otimes R}^{-1} (z\pi(a)).\]
  Notice that 
\[\big((\text{id}_{M\otimes R} \otimes p)\circ(\text{id}_M \otimes \epsilon)\circ(\text{id}_M\otimes \text{id}_R\otimes T_p)\big) = \text{id}_M \otimes f_p: M\otimes (R\otimes pRp) \rightarrow M\otimes (R\otimes pRp)\] 
where 
\[f_p = (\text{id}_R \otimes p)\circ\epsilon \circ (\text{id}_R \otimes T_p): R \otimes pRp \rightarrow R\otimes pRp\]
 is a unital $*$-homomorphism.  Thus we have 
\[f_p(x) =\lim_j a_jxa_j^*\] 
where $a_j \in \mathcal{U}(R\otimes pRp)$.  Note that $a_j^*a_j=a_ja_j^* = 1_R\otimes p$.  So we have 
\[\big((\text{id}_{M\otimes R} \otimes p)\circ(\text{id}_M \otimes \epsilon)\circ(\text{id}_M\otimes \text{id}_R\otimes T_p)\big) (y) = \lim_j (1_M \otimes a_j)y(1_M \otimes a_j)^*.\]

Similarly, \[\rho_2(a) \otimes p^{\perp} = \big((\text{id}_{M\otimes R} \otimes p^{\perp})\circ(\text{id}_M \otimes \epsilon) \circ (\text{id}_M \otimes \text{id}_R \otimes T_p^{\perp})\big) \circ Ad(v+w) \circ \sigma_{M\otimes R}^{-1} (z^{\perp}\pi(a)),\]
and there is a sequence $\left\{b_j\right\} \subset R\otimes R$ with $b_j^*b_j = b_jb_j^* = 1\otimes p^{\perp}$ such that
\[\big((\text{id}_{M\otimes R} \otimes p^{\perp})\circ(\text{id}_M \otimes \epsilon) \circ (\text{id}_M \otimes \text{id}_R \otimes T_p^{\perp})\big)(y) = \lim_j (1_M\otimes b_j)y(1_M\otimes b_j)^*.\]

We now have that $(1_M\otimes a_j + 1_M\otimes b_j)$ is a unitary for every $j$ and so 
\[\rho_1(a)\otimes p + \rho_2(a)\otimes p^{\perp} \]
\begin{align*}
&= \lim_j Ad(1_M\otimes a_j + 1_M \otimes b_j)\circ Ad(v+w)\circ\sigma_{M\otimes R}^{-1}(z \pi(a) + z^{\perp}\pi(a))\\
& = \lim_j Ad(1_M\otimes a_j + 1_M \otimes b_j)\circ Ad(v+w)\circ\sigma_{M\otimes R}^{-1} (\pi(a)).
\end{align*}

Thus,
\begin{align*}
[t[\rho_1] + (1-t)[\rho_2] & = [\sigma_{M\otimes R}(\rho_1 \otimes p + \rho_2\otimes p^{\perp})]\\
& =  [\sigma_{M\otimes R} \circ Ad(v+w) \circ \sigma_{M\otimes R}^{-1} \circ \pi]\\
& =  [\sigma_{M\otimes R} \circ \sigma_{M\otimes R}^{-1} \circ \pi]\\
& =  [ \pi].
\end{align*}  So (1) has been verified.

\noindent (2): To show that $[\pi] \neq [\rho_1]$ we will exhibit an element $x$ with $||\pi(x)||_2 \neq ||\rho_1(x)||_2.$  Let $\varepsilon>0$ be such that $\ds \varepsilon < \frac{\sqrt{1-t}}{1+\frac{1}{\sqrt{t}}}$.  Let $x \in \FR{A}$ be such that $||\pi(x) - (1-z)||_2 < \varepsilon$.  Then we have 
\begin{equation}\label{zpix}
\begin{aligned}
||z\pi(x)||_2 &= ||z\pi(x) + z(1-z)||_2\\
 &= ||z(\pi(x) - (1-z))||_2 \\
&\leq ||z||\cdot ||\pi(x) - (1-z)||_2\\
 &< \varepsilon
\end{aligned}
\end{equation}
and
\begin{equation}\label{pix}
\begin{aligned}
||\pi(x)||_2 &= ||(\pi(x) - (1-z)) + (1-z)||_2\\
& \geq ||(1-z)||_2 - ||\pi(x) - (1-z)||_2\\
& > \sqrt{1-t} - \varepsilon.
\end{aligned}
\end{equation}

Let $\varphi: z (M \otimes R)z \rightarrow M\otimes R$ be given by
\[\varphi = (\text{id}_M \otimes \epsilon) \circ (\text{id}_M\otimes \text{id}_R \otimes T_p) \circ Ad(v+w) \circ \sigma_{M\otimes R}^{-1}.\]
So $\rho_1(x) = \varphi(z\pi(x))$. If $[\pi] = [\rho_1]$ then there exists a sequence of unitaries $\left\{u_n\right\}\subset \mathcal{U}(M\otimes R)$ such that 
\[\lim_{n\rightarrow \infty} u_n \pi(x) u_n^* = \rho_1(x) = \varphi(z\pi(x)).\]
So we have the following equation of norms
\[||\pi(x)||_2 = ||\lim_{n\rightarrow \infty} u_n \pi(x) u_n^*||_2 = ||\varphi(z\pi(x))||_2.\]
Then according to \eqref{pix} we have on one hand 
\[||\pi(x)||_2 > \sqrt{1-t}-\varepsilon,\]
while on the other hand, by \eqref{zpix} we have
\[||\varphi(z\pi(x))||_2 = \frac{1}{\sqrt{t}}||z\pi(x)||_2 < \frac{\varepsilon}{\sqrt{t}}.\]
This gives the following implications
\begin{align*}
\sqrt{1-t} - \varepsilon < \frac{\varepsilon}{\sqrt{t}} & \Rightarrow  \sqrt{1-t} < \Big(1+\frac{1}{\sqrt{t}}\Big)\varepsilon\\
 &\Rightarrow  \frac{\sqrt{1-t}}{1+\frac{1}{\sqrt{t}}} < \varepsilon.
\end{align*} 
This last line is a contradiction to the assumption that $\ds \varepsilon < \frac{\sqrt{1-t}}{1+\frac{1}{\sqrt{t}}}$.  So we can conclude that $[\pi]\neq [\rho_1]$, and this completes the proof.
\end{proof}

\begin{exmpl}\label{sufffail}
We will show that the converse of the above theorem does not hold in general.  Let $M$ be a McDuff factor with an automorphism $\beta \in \text{Aut}(M)\setminus \overline{\text{Inn}}(M)$. As in Example \ref{injfail} we may consider the nonhyperfinite McDuff factor $\otimes L(\mathbb{F}_2)$; the fact that the inner automorphisms are not dense in the automorphism group of $M$ is the main player in this argument.  Let $\FR{A}$ be a dense $C^*$-subalgebra of $M$.  Let $\rho_1 = \text{id}_{\FR{A}}$ be the identity inclusion of $\FR{A}$ in $M$, and let $\rho_2 = \beta \circ \rho_1$.  Then from Example \ref{injfail} we have that $[\rho_1]\neq [\rho_2]$.  Thus $[\pi] = \frac{1}{2}[\rho_1] + \frac{1}{2}[\rho_2]$ is not an extreme point in $\HOM_w(\FR{A},M)$.  Let $p$ be a projection in $R$ so that $\tau(p) = \frac{1}{2}$, and thus $\sigma_M(\rho_1\otimes p + \rho_2 \otimes p^{\perp})$ is a representative of $[\pi]$.  We will show that $W^*((\rho_1\otimes p + \rho_2 \otimes p^{\perp})(\FR{A})) \cong W^*(\FR{A}) \cong M$, and thus giving an example of a non-extreme point with a factorial closure of its image.  Consider the map \[\varphi: \FR{A} \rightarrow (\rho_1 \otimes p + \rho_2 \otimes p^{\perp})(\FR{A})\] given by 
\begin{align*} 
\varphi(a) & =  \rho_1(a) \otimes p + \rho_2(a) \otimes p^{\perp}\\
& =  a\otimes p + \beta(a) \otimes p^{\perp}.
\end{align*}
The map $\varphi$ is clearly a bijective $*$-homomorphism.  The following computation shows that $\varphi$ is also isometric with respect to $||\cdot||_2$.
\begin{align*}
||\varphi(a)||_2^2 & =  \tau(( a\otimes p + \beta(a) \otimes p^{\perp})^*( a\otimes p + \beta(a) \otimes p^{\perp}))\\
& =  \tau(a^*a\otimes p + \beta(a^*a)\otimes p^{\perp})\\
& =  \frac{1}{2}\tau(a^*a) + \frac{1}{2}\tau(a^*a)\\
& =  ||a||_2^2.  
\end{align*}
Evidently, $\varphi$ extends to an isomorphism between $M =W^*(\FR{A})$ and $W^*((\rho_1\otimes p + \rho_2 \otimes p^{\perp})(\FR{A}))$.
\end{exmpl}

While the converse of Theorem \ref{extfact} fails in general, if we combine Theorem \ref{nuctrace} with the observation that a trace is extreme if and only if it gives a factorial GNS construction, then the converse holds in the nuclear case.  Thus we have the following theorem characterizing extreme points of $\HOM_w(\FR{A},M)$ when $\FR{A}$ is nuclear.

\begin{thm}
If $\FR{A}$ is nuclear, then $[\pi] \in \HOM_w(\FR{A},M)$ is extreme if and only if $W^*(\pi(\FR{A}))$ is a factor.
\end{thm}

\noindent We will extend this characterization in \S\S\ref{amenable1st}, and we will see in \S\S\ref{amenable2nd} that this characterization of extreme points in $\HOM_w(\FR{A},M)$ holds for a general $\FR{A}$ when $M = R$.

\subsection{Quotients}

We have access to quotients of $C^*$-algebras; this access is unavailable in the setting of \cite{topdyn} because II$_1$ factors are simple.  Given a $*$-homomorphism $h: \FR{A} \rightarrow M$, let $q_h: \FR{A} \rightarrow \FR{A}/\text{ker}(h)$ be the canonical quotient map; and let $\overline{h}: \FR{A}/\text{ker}(h) \rightarrow M$ be the natural $*$-homomorphism that makes the following diagram commute.
\begin{equation}
\begin{tikzpicture}[baseline=(current bounding box.base)]\notag
	\matrix(m)[matrix of math nodes, row sep=2em,column sep = 2em,minimum width=2em]
		{\FR{A} & &  M\\
			& & \\
		 & \FR{A}/\text{ker}(h)& \\};
	\path[-stealth]
		(m-1-1) edge node [above] {$h$} (m-1-3)
		(m-1-1) edge node [below] {$q_h$} (m-3-2)
		(m-3-2) edge node [below] {$\overline{h}$} (m-1-3);
\end{tikzpicture}
\end{equation}
In particular $h = \overline{h} \circ q_h$.  The map $q_h$ induces a map $q^*_h: \HOM_w(\FR{A}/\text{ker}(h),M) \rightarrow \HOM_w(\FR{A},M)$ (see Proposition \ref{phihat}) with \begin{equation}q^*_h([\overline{h}]) = [\overline{h} \circ q_h] = [h].\label{quotienthat}\end{equation}

\begin{dfn}
A \emph{singly exposed face} of a closed bounded convex subset $C$ of a Banach space is a face that can be described as  $\left\{x \in C: h(x) = M\right\}$ where $h: C \rightarrow \mathbb{R}$ is a continuous affine functional and $M = \max \left\{h(x): x \in C\right\}$.
\end{dfn}  

\begin{dfn}
Given $a \in \FR{A}$, we define a natural continuous affine functional $f_a$ on $\HOM_w(\FR{A},M)$ given by $f_a([\pi]) = \tau_M(\pi(a))$.
\end{dfn}

We leave it to the reader to check that this satisfies the definition of a continuous affine functional.

\begin{lem}\label{kernelface}
 Let $J \trianglelefteq \FR{A}$ be a closed, two-sided ideal, and let $q: \FR{A} \rightarrow \FR{A}/J$ be the canonical quotient map.  Then we have that $q^*(\HOM_w(\FR{A}/J,M))$ is a singly exposed face of $\HOM_w(\FR{A},M)$.  
\end{lem}

\begin{proof}

Without loss of generality, assume that $J$ is generated by $\left\{a_n\right\}_{n=1}^\infty \subset (J)^+_{\leq 1}$.  Put \[a := \sum_{n=1}^\infty \frac{1}{2^n}a_n.\]  Consider $f_{-a}$.  By the positivity of $\tau_M$ we get that $f_{-a}([\pi]) \leq 0$ for every $[\pi] \in \HOM_w(\FR{A},M)$. It is a quick observation to see that \[q^*(\HOM_w(\FR{A}/J,M)) = \left\{ [\pi] \in \HOM_w(\FR{A},M): f_{-a}([\pi]) = 0\right\}.\qedhere\]  

%Let $[\pi] \in \HOM_w(\FR{A}/J,M)$ and say \[q^*([\pi]) = t[\rho_1] + (1-t)[\rho_2]\] for $t \in (0,1)$ and $\rho_i \in \HOM_w(\FR{A},M), i = 1,2$.  We want to show that $[\rho_i] \in q^*(\HOM_w(\FR{A}/J,M))$ for $i = 1,2$.
%
%We have \[q^*([\pi]) = [\pi \circ q] = [\sigma_M(\rho_1 \otimes p + \rho_2 \otimes p^{\perp})]\] for $p \in \mathcal{P}(R)$ with $\tau(p) = t$; so it follows that 
%\begin{equation}\label{kinc}
%J \subseteq \text{ker}(\pi \circ q) \subseteq \text{ker}(\rho_i)
%\end{equation} 
%for $i=1,2$.  Let $q_{\rho_i}$ denote the canonical quotient map \[q_{\rho_i}: \FR{A} \rightarrow \FR{A}/\text{ker}(\rho_i)\] for each $i=1,2$. By \eqref{kinc}, for each $i = 1,2$ there exists a well-defined $*$-homomorphism $\varphi_i: \FR{A}/J\rightarrow \FR{A}/\text{ker}(\rho_i)$ given by \[\varphi_i(q(a)) = q_{\rho_i}(a)\] for every $a \in \FR{A}$.  This immediately gives that for $i=1,2$, \[q^*_{\rho_i} = q^* \circ \varphi^*_i.\]  So by \eqref{quotienthat} we have \[[\rho_i] = q^*_{\rho_i}([\overline{\rho_i}]) = q^*(\varphi^*_i([\overline{\rho_i}]))\] for $i = 1,2$.
\end{proof}

%So the faces of $\HOM_w(\FR{A},M)$ coming from ideals of $\FR{A}$ are necessarily exposed faces.

The quotient map $q: \FR{A} \rightarrow \FR{A}/J$ is surjective, so by Proposition \ref{surin} we get that $q^*$ is a homeomorphism onto its image.  Thus we may regard $\HOM_w(\FR{A}/J,M)$ as a face of $\HOM_w(\FR{A},M)$ by identifying it with its image $q^*(\HOM_w(\FR{A}/J,M))$.  So every quotient of $\FR{A}$ gives a face of $\HOM_w(\FR{A},M)$.  Conversely, by Example \ref{blackadar}, any metrizable Choquet simplex arises as $\HOM_w(\FR{A},M)$ for some simple AF-algebra; so in this situation, any face of $\HOM_w(\FR{A},M)$ is not induced by a quotient of $\FR{A}$. 

We can use this discussion to give a sufficient condition for the existence of extreme points in $\HOM_w(\FR{A},M)$.

\begin{thm}\label{nucquot}
If $\FR{A}$ has a finite nuclear quotient, then $\HOM_w(\FR{A},M)$ has extreme points.
\end{thm}

\begin{proof}
If $\FR{A}/J$ is finite and nuclear, then $\HOM_w(\FR{A}/J,M) \cong T(\FR{A})$; and thus \linebreak $\HOM_w(\FR{A}/J,M)$ has extreme points.  An extreme point of a face is extreme.
\end{proof}

%\begin{cor}\label{extremekernel}
%Given $[\pi] \in \HOM_w(\FR{A},M)$, we have that $[\pi]$ is extreme if and only if $[\overline{\pi}] \in \HOM_w(\FR{A}/\text{ker}(\pi), M)$ is extreme.
%\end{cor}
%
%\begin{proof}
%It is clear that the kernel operation is constant on the equivalence class of $\pi$.  It is also clear that if $[\pi] = [\rho]$ then $[\overline{\pi}] = [\overline{\rho}]$.  So there are no issues with well-definition in the statement of the lemma, and the statement follows from Lemma \ref{kernelface} and basic convex geometry.\end{proof}

%\begin{rmk}\label{universal}
%By universality, any $\FR{A}$ is a quotient of $C^*(\mathbb{F}_\infty)$. So for any $\FR{A}$, an affinely homeomorphic copy of $\HOM_w(\FR{A},M)$ appears as a face in $\HOM_w(C^*(\mathbb{F}_\infty),M)$.
%\end{rmk}

If we combine Lemma \ref{kernelface} with the observation that any $\FR{A}$ is a quotient of $C^*(\mathbb{F}_\infty)$, we get the following theorem.

\begin{thm}\label{uniface}
For any $\FR{A}$, an affinely homeomorphic copy of $\HOM_w(\FR{A},M)$ appears as a face of $\HOM_w(C^*(\mathbb{F}_{\infty}),M)$.
\end{thm}

\begin{rmk}
% In light of this last corollary, it seems worthwhile to study the extreme points of $\HOM_w(C^*(\mathbb{F}_{\infty}),M)$ in particular.  
Theorem \ref{nucquot} guarantees that $\HOM_w(C^*(\mathbb{F}_{\infty}),M)$ has extreme points.
\end{rmk}

\section{Ultrapower Situation}\label{ultrasit}

Let $\mathcal{U}$ denote a free ultrafilter on $\mathbb{N}$. We now take the opportunity to extend Brown's construction of the convex structure on $\HOM(N,R^\mathcal{U})$ to a convex structure on $\HOM(\FR{A},M^\mathcal{U})$: the space of unital $*$-homomorphisms $\FR{A} \rightarrow M^\mathcal{U}$ modulo unitary equivalence.  This space has the same metric as its predecessor.  For $\sigma_M: M\otimes R \rightarrow M$ a regular isomorphism, we let \[(\sigma_M)^\mathcal{U}: (M \otimes R)^\mathcal{U} \rightarrow M^\mathcal{U}\] be given by \[(\sigma_M)^\mathcal{U} ((x_i)_\mathcal{U}) = (\sigma_M(x_i))_\mathcal{U}.\]   There is a natural way to embed $M$ into $M^\mathcal{U}$ as cosets of constant sequences: $x \in M \mapsto (x)_\mathcal{U} \in M^\mathcal{U}$. Also, as in Remark 3.2.4 of \cite{topdyn}, there is a natural embedding of $M^\mathcal{U} \otimes R^\mathcal{U}$ in $(M \otimes R)^\mathcal{U}$.  So for $x \in M^\mathcal{U}$ and $y \in R^\mathcal{U}$, the expression $(\sigma_M)^\mathcal{U}(x\otimes y)$ makes sense once we consider $x \otimes y \in (M \otimes R)^\mathcal{U}$ via this natural embedding. We can now define convex combinations in $\HOM(\FR{A},M^\mathcal{U})$ in a way similar to Definition \ref{convexcombo}.

\begin{dfn}
For $[\pi],[\rho] \in \HOM(\FR{A},M^\mathcal{U})$ and $t \in [0,1]$, define 
\begin{equation}
t[\pi] + (1-t)[\rho] = [\sigma_M^\mathcal{U}(\pi \otimes (p)_\mathcal{U} + \rho \otimes (p^\perp)_\mathcal{U})]
\end{equation}
where $\sigma_M: M\otimes R \rightarrow M$ is a regular isomorphism and $p \in R$ is a projection with $\tau(p) = t$.
\end{dfn}

\noindent We leave it to the reader to check that this convex combination is well-defined and satisfies the axioms of a convex-like structure (the proofs are analogous).  This generalization of $\HOM(N,R^\mathcal{U})$ from \cite{topdyn} also retains the characterization of extreme points.  One can see this by following the same reasoning as in \cite{topdyn} and looking at \tql cut-downs\tqr of homomorphisms by projections in the relative commutant.  We provide a definition of a cut-down and the ingredients for establishing this characterization in the general setting without proof -- most of the arguments follow from direct analogy.

\begin{dfn}
For technical reasons, we consider homomorphisms $\pi: \FR{A} \rightarrow (M\otimes R)^\mathcal{U}$.  Given $\pi: \FR{A} \rightarrow (M\otimes R)^\mathcal{U}$, let $q \in \pi(\FR{A})' \cap (M\otimes R)^\mathcal{U}$ be a projection.  We now define the cut-down \[\pi_q: \FR{A} \rightarrow (M \otimes R)^\mathcal{U}\] of $\pi$ by $q$.  Let $p \in R$ be a projection in $R$ with $\tau(p) = \tau(q)$.  Let $v,w \in (M\otimes R\otimes R)^\mathcal{U}$ be a partial isometries such that 
\begin{align*}
v^*v &= (\sigma_{(M\otimes R)}^\mathcal{U})^{-1}(q), & vv^* &= (1 \otimes 1 \otimes p)_\mathcal{U},\\
w^*w &= (\sigma_{(M\otimes R)}^\mathcal{U})^{-1}(q^\perp), & ww^*&=(1 \otimes 1 \otimes p^\perp)_\mathcal{U}.
\end{align*} 
Put $u:= v+w$. Let $\theta_p: pRp \rightarrow R$ be an isomorphism.  Then we let \[\pi_q(\cdot) := (\sigma_{(M\otimes R)}^\mathcal{U})\circ (\text{id}_M\otimes \text{id}_R \otimes \theta_p)^\mathcal{U}\circ (\operatorname{Ad}u)\circ(\sigma_{(M\otimes R)}^\mathcal{U})^{-1}(q\pi(\cdot)).\] 
\end{dfn}

\begin{prop}
Using the above definition of cut-downs, one can verify the following six statements.

\begin{enumerate}

	\item $[\pi_p]$ is independent of choices made in the definition

	\item $[\pi_p] = [(\text{Ad}u \circ \pi)_{u pu^*}]$

	%\item Given any $t \in (0,1)$ there exists a projection $p \in \pi(N)'\cap M^\mathcal{U}$ with trace $t$ so that $[\pi]=[\pi_p]$.

	\item Given projections $p,q \in \pi(N)'\cap M^\mathcal{U}$ with $\tau(p) = \tau(q)$, we have that $[\pi_p] = [\pi_q]$ if and only if $p$ is Murray-von Neumann equivalent to $q$ in $\pi(N)'\cap M^\mathcal{U}$.

	\item If $[\pi] = t[\rho_1] + (1-t)[\rho_2]$ then there is a projection $p \in \pi(N)'\cap M^\mathcal{U}$ with trace $t$ such that $[\rho_1] = [\pi_p]$ and $[\rho_2] = [\pi_{p^\perp}]$.

	\item Given any $p \in \pi(N)'\cap M^\mathcal{U}, [\pi] = \tau(p) [\pi_p]  + \tau(p^\perp)[\pi_{p^\perp}]$.

	\item A finite diffuse von Neumann algebra $A$ is a factor if and only if for projections $p,q \in A, \tau(p) = \tau(q) \Rightarrow p$ is Murray-von Neumann equivalent to $q$.

\end{enumerate}

\end{prop}

\noindent With the above proposition established, we have the following theorem.

\begin{thm}\label{ultraextreme}
The equivalence class $[\pi] \in \HOM(\FR{A},M^\mathcal{U})$ is extreme if and only if $\pi(\FR{A})'\cap M^\mathcal{U}$ is a factor.
\end{thm}

\begin{rmk}\label{existence1}
The existence of extreme points in $\HOM(\FR{A},M^\mathcal{U})$ remains an open problem.  The existence of extreme points in the context of $\HOM(N,R^\mathcal{U})$ as in \cite{topdyn} is a well-known open question.  The most recent work done on this question can be found in \cite{caplup}.
\end{rmk}

Given any McDuff $M$ and any separable finite hyperfinite factor $N$, we have that any embedding of $N$ into $M^\mathcal{U}$ is unique up to unitary equivalence (Corollary 3.4, \cite{autoultra}).  That is, $\HOM(N,M^\mathcal{U})$ is a single point.  Combining this observation with the above theorem gives the following consequence.

\begin{cor}\label{finitefactorcommutant}
For any McDuff $M$ and any finite hyperfinite factor $N$, any embedding $\pi: N \rightarrow M^\mathcal{U}$ has the property that its relative commutant $\pi(N)'\cap M^\mathcal{U}$ is a factor.
\end{cor}

\noindent Thanks to an observation by S. White, a more general version of the above corollary is available.  In particular, we do not need to require that $M$ is McDuff.

\begin{thm}\label{stuart}
For any separable II$_1$-factor $X$ and any separable finite hyperfinite factor $N$, any embedding $\pi: N \rightarrow X^\mathcal{U}$ has the property that its relative commutant $\pi(N)'\cap X^\mathcal{U}$ is a factor.
\end{thm}

\begin{proof}
Throughout this proof we will abuse notation by letting $\tau$ denote both the trace on $N$ and the trace on $X^\mathcal{U}$. The proof of this theorem essentially follows from Lemma 3.21 of \cite{bbstww}.  The lemma says, among other things, that the collection of tracial states \[\left\{\tau(\pi(x)\cdot): x \in N^+, \tau(x) = 1\right\}\] on $\pi(N)'\cap X^\mathcal{U}$ is weak-$*$ dense in the trace space of $\pi(N)'\cap X^\mathcal{U}$.  For any fixed $x \in N^+$ with $\tau(x) = 1$, Dixmier approximation gives that for any $\varepsilon > 0$, there exist unitaries $u_1, \dots, u_n \in \mathcal{U}(N)$ and numbers $0\leq \lambda_1,\dots, \lambda_n \leq 1$ with $\sum \lambda_j = 1$ such that \[\Big|\Big|\sum_{j=1}^n \lambda_j u_j x u_j^* - 1_N\Big|\Big| < \varepsilon.\]  Note that for such an $x \in N, \tau(\pi(x)\cdot) = \tau(\pi(\sum_j \lambda_j u_jx u_j^*) \cdot)$.  

We will now show that $\pi(N)' \cap X^\mathcal{U}$ is a factor by showing that it has a unique normalized trace (in particular, the trace induced by the unique trace on $X^\mathcal{U}$).  Let $T$ be a tracial state on $\pi(N)'\cap X^\mathcal{U}$.  Fix $\delta > 0$ and $y \in \pi(N)'\cap X^\mathcal{U}$ with $||y||\leq 1$.  By Lemma 3.21 of \cite{bbstww} there is an $x_y \in N^+$ with $\tau(x_y) =1$ such that \[|T(y) - \tau(\pi(x_y)y)| < \frac{\delta}{2}.\]  For such an $x_y$, let $u_1, \dots, u_n \in \mathcal{U}(N)$ and $0\leq \lambda_1,\dots, \lambda_n \leq 1$ with $\sum_j \lambda_j = 1$ be such that \[\Big|\Big|\sum_{j=1}^n \lambda_j u_j x_y u_j^* - 1_N\Big|\Big| < \frac{\delta}{2}.\]  Then we have 
\begin{align*}
\Big|\tau\Big(\pi\Big(\sum_j \lambda_j u_j x_y u_j^*\Big)y\Big) - \tau(y)\Big| &= \Big|\tau\Big(\pi\Big(\sum_j \lambda_j u_j x_y u_j^* - 1_N\Big)y\Big)\Big|\\
&\leq \Big|\Big|\pi\Big(\sum_j \lambda_j u_j x_y u_j^* - 1_N\Big)y\Big|\Big|\\
&\leq \Big|\Big|\sum_j \lambda_j u_j x_y u_j^* - 1_N\Big|\Big|\cdot ||y||\\
& < \frac{\delta}{2}.
\end{align*}  
Thus,
\begin{align*}
|T(y) - \tau(y)| &\leq |T(y) - \tau(\pi(x_y)y)| + |\tau(\pi(x_y)y) - \tau(y)|\\
& =  |T(y) - \tau(\pi(x_y)y)| + \Big|\tau\Big(\pi\Big(\sum_j \lambda_j u_jx_yu_j^*\Big)y\Big) - \tau(y)\Big|\\
& < \delta.\qedhere
\end{align*}
\end{proof}

\noindent The case where $N=X=R$ is already well-known, but we have not seen the statement as it appears above in the literature.  A similar discussion does appear in Section 2 of \cite{popa} (in particular Theorem 2.1 and Conjecture 2.3.1). There, Popa addresses bicentralizers in ultraproduct von Neumann algebras, while we are concerned with relative commutants (centralizers).  

Corollary 5.3 of \cite{topdyn} says \tql $R$ is the unique separable II$_1$-factor with the property that every embedding into $R^\mathcal{U}$ has factorial commutant.\tqr \, Theorem \ref{stuart} leads us to the following stronger version of Brown's statement.

\begin{thm}\label{Rchar}
Let $N$ be an embeddable separable II$_1$-factor.  The following are equivalent:
\begin{enumerate}
\item $N\cong R$;
\item For any separable II$_1$-factor $X$ and any embedding $\pi: N \rightarrow X^\mathcal{U}$, $\pi(N)'\cap X^\mathcal{U}$ is a factor;
\item For any separable II$_1$-factor $X$ and any embedding $\pi: N\rightarrow X^\mathcal{U}$, the collection of tracial states $\left\{\tau(\pi(x)\cdot): x \in N^+, \tau(x) = 1\right\}$  is weak-$*$ dense in the trace space of $\pi(N)'\cap X^\mathcal{U}$.
\end{enumerate}
\end{thm}

\begin{proof}
((1)$\Rightarrow$ (3)):  This follows from Lemma 3.21 of \cite{bbstww}.

\noindent ((3)$\Rightarrow$ (2)): This follows from the proof of Theorem \ref{stuart}.

\noindent ((2)$\Rightarrow$ (1)): This follows from Corollary 5.3 of \cite{topdyn}.
\end{proof}

\noindent Note that the above theorem implies that for any $y \in \pi(R)'\cap X^\mathcal{U}, \tau(y) = \tau(\pi(x)y)$ for every $x \in R^+$ with $\tau(x) = 1$.

The characterization in Theorem \ref{ultraextreme} of extreme points in the ultrapower case reveals some information on extreme points in the separable $\HOM_w(\FR{A},M)$ setting.  Using the canonical constant-sequence embedding of $M$ into $M^\mathcal{U}$ we get the following map.

\begin{dfn}\label{betadef}
Let \[\beta_{(\FR{A},M)}: \HOM_w(\FR{A}, M) \rightarrow \HOM(\FR{A},M^{\mathcal{U}})\] be given by \[\beta_{(\FR{A},M)}([\pi]) = [\pi^\mathcal{U}]\] where $\pi^\mathcal{U}(a) = (\pi(a))_\mathcal{U}$ for $a \in \FR{A}$.
\end{dfn}

\begin{prop}\label{betainj}
The map $\beta_{(\FR{A},M)}$ is a well-defined affine homeomorphism onto its image.
\end{prop}

\begin{proof}
That $\beta_{(\FR{A},M)}$ is continuous and affine is an easy check.  Well-defined and injective follow from Theorem 3.1 of \cite{autoultra}.

It remains to show that if $\beta_{(\FR{A},M)}([\pi_n]) \rightarrow \beta_{\FR{A},M)}([\pi])$ in $\HOM(\FR{A},M^\mathcal{U})$ then $[\pi_n] \rightarrow [\pi]$ in $\HOM_w(\FR{A},M)$.  So suppose that \[[\pi_n^\mathcal{U}] \rightarrow [\pi^\mathcal{U}].\]  That means that there exists homomorphisms $\varphi_n \in [(\pi_n)^\mathcal{U}]$ such that \[\varphi_n(a) \rightarrow (\pi)^\mathcal{U}(a)\] for every $a \in \FR{A}$ in the trace norm on $R^\mathcal{U}$.  Now for $n$ fixed, $\varphi_n \in [(\pi_n)^\mathcal{U}]$ means that there is a unitary $u_n \in \mathcal{U}(R^\mathcal{U})$ such that $\varphi_n(a) = u_n \pi_n(a) u_n^*$.  Without loss of generality, say that $u_n = (u_{n,j})_\mathcal{U}$.  So we have that $\varphi_n(a) = (u_{n,j} \pi_n(a) u_{n,j}^*)_\mathcal{U}$.  

We follow an argument similar to the one found in the proof of Lemma \ref{unitarysubseq}. Let $\left\{a_i\right\} \subset \FR{A}_{\leq 1}$ be dense in the unit ball of $\FR{A}$.  Recursively construct a sequence of positive integers $\left\{N_k\right\}$ in the following way.  Let $N_1 \in \mathbb{N}$ be such that for every $n \geq N_1$, \[||\varphi_n(a_1) - \pi(a_1)||_2 < \frac{1}{2}.\] Let $N_2 > N_1$ be such that for every $n \geq N_2$ and $i =1,2$ \[||\varphi_n(a_i) -\pi(a_i)||_2 < \frac{1}{4}.\]  In general, let $N_k>N_{k-1}$ be such that for every $n \geq N_k$ and $1 \leq i \leq k$  \[||\varphi_n(a_i) - \pi(a_i)||_2 < \frac{1}{2k}.\]

For $N_k \leq n < N_{k+1}$ and $1\leq i \leq k$, let \[L(n,i) = ||\varphi_n(a_i) - \pi(a_i)||_2 = \lim_{j\rightarrow \mathcal{U}}||u_{n,j}\pi_n(a_i)u_{n,j}^* - \pi(a_i)||_2.\]  By our construction, we have $0\leq L(n,i) < \frac{1}{2k}$.  By definition \[\left\{j: \big| ||u_{n,j}\pi_n(a_i)u_{n,j}^*||_2 - L(n,i)\big| < \frac{1}{2k}\right\} \in \mathcal{U}.\]  And since \[\left\{j: \big| ||u_{n,j}\pi_n(a_i)u_{n,j}^*||_2 - L(n,i)\big| < \frac{1}{2k}\right\} \subseteq \left\{j: ||u_{n,j}\pi_n(a_i)u_{n,j}^* - \pi(a_i)||_2 < \frac{1}{k}\right\},\] we get \[\left\{j: ||u_{n,j}\pi_n(a_i)u_{n,j}^* - \pi(a_i)||_2 < \frac{1}{k}\right\} \in \mathcal{U}.\]  So for $N_k \leq n <N_{k+1}$, the intersection \[A_n:= \bigcap_{1\leq i \leq k} \left\{j: ||u_{n,j}\pi_n(a_i)u_{n,j}^* - \pi(a_i)||_2 < \frac{1}{k}\right\}\] is in the ultrafilter $\mathcal{U}$, and hence is nonempty.  Pick $j(1) \in A_1$ and for $n>1$, let $j(n) \in A_n \cap \left\{j > j(n-1)\right\}$ (also nonempty since it is an element of the ultrafilter $\mathcal{U}$).

Let $v_n = u_{n,j(n)}$.  We will now show that for $a \in \FR{A}_{\leq 1}, v_n\pi_n(a)v_n^* \rightarrow \pi(a)$.  Fix $\varepsilon > 0$.  Let $i' \in \mathbb{N}$ be such that $\ds ||a-a_{i'}|| < \frac{\varepsilon}{4}$.  Let $k \in \mathbb{N}$ be such that $i' \leq k$ and $\ds \frac{1}{k} < \frac{\varepsilon}{2}$.  Let $n \geq N_k$; thus $n \in [N_{k+c}, N_{k+c+1})$ for some $c \geq 0$.  Then \[j(n) \in \bigcap_{1\leq i \leq k+c}\left\{j: ||u_{n,j} \pi_n(a_i) u_{n,j}^* - \pi(a_i)||_2 < \frac{1}{k+c}\right\}.\]  So we have
\begin{align*}
||v_n\pi_n(a)v_n^* - \pi(a)||_2  \leq&\;  ||v_n\pi_n(a)v_n^* - v_n\pi_n(a_{i'})v_n^*||_2 + ||v_n\pi_n(a_{i'})v_n^* - \pi(a_{i'})||_2 \;+\\
&\;   ||\pi(a_{i'}) - \pi(a)||_2\\
 < &\; \frac{\varepsilon}{4} + ||u_{n,j(n)} \pi_n(a_{i'}) u_{n,j(n)}^* - \pi(a_{i'})||_2 + \frac{\varepsilon}{4}\\
 <&\;\frac{\varepsilon}{4} + \frac{1}{k+c} + \frac{\varepsilon}{4}\\
 < &\; \varepsilon.
\end{align*}

Since $(\operatorname{Ad}(v_n) \circ \pi_n)\in [\pi_n]$, we have $[\pi_n] \rightarrow [\pi]$.  Thus $\beta_{(\FR{A},R)}$ is an affine homeomorphism onto its image.
\end{proof}

A natural follow-up question would be: \tql is $\beta_{(\FR{A},M)}$ ever surjective?\tqr  The answer is: sometimes.  The following example will show that $\beta_{(\FR{A},M)}$ is not surjective in general.

\begin{exmpl}
This example will show that $\HOM(\FR{A},M^\mathcal{U})$ can strictly contain $\HOM_w(\FR{A},M)$; or in the notation of Definition \ref{betadef}, $\beta_{(\FR{A},M)}$ can fail to be surjective.  Let $\FR{A} = C^*(\mathbb{F}_\infty)$ and $M =R$.  Furthermore, let $N$ be a non-hyperfinite separable embeddable II$_1$-factor.  By Proposition 3.1 of \cite{con}, $\FR{A}$ can be embedded into $N$ (say via $\zeta: \FR{A} \rightarrow N$) such that it is weakly dense in $N$.  Consider the map $\zeta^*: \HOM(N, R^\mathcal{U}) \rightarrow \HOM(\FR{A},R^\mathcal{U})$ given by \[\zeta^*([\pi]) = [\pi\circ \rho].\]  Just as in Proposition \ref{phihat}, $\zeta^*$ is well-defined, continuous, and affine.  It is not hard to see that $\zeta^*$ is additionally injective.  Now, by Theorem A.1 of \cite{topdyn}, we know that $\HOM(N,R^\mathcal{U})$ is nonseparable.  So, since $\zeta^*(\HOM(N,R^\mathcal{U})) \subset \HOM(\FR{A},R^\mathcal{U})$, this means that $\HOM(\FR{A},R^\mathcal{U})$ is nonseparable.  On the other hand, $\HOM_w(\FR{A},R)$ is separable.  So by cardinality considerations, \[\beta_{(\FR{A},R)}(\HOM_w(\FR{A},R)) \subsetneq \HOM(\FR{A},R^\mathcal{U}).\]
\end{exmpl}

With this embedding $\beta_{(\FR{A},M)}$ established, we immediately get a sufficient condition for extreme points in $\HOM_w(\FR{A},M)$.

\begin{thm}\label{comext}
For $\pi: \FR{A}\rightarrow M$, if $\pi^\mathcal{U}(\FR{A})' \cap M^\mathcal{U}$ is a factor, then $[\pi] \in \HOM_w(\FR{A},M)$ is extreme.
\end{thm}

\noindent The converse of the above statement would be true if we can show that \linebreak $\beta_{(\FR{A},M)}(\HOM_w(\FR{A},M))$ is a face of $\HOM(\FR{A},M^\mathcal{U})$.  This question comes down to asking if the cut-down of a constant-sequence homomorphism $\pi^\mathcal{U}$ is itself a constant-sequence homomorphism.

%So far we have a necessary condition and a sufficient condition but not a necessary and sufficient condition. Combining the above theorem with Theorem \ref{extfact} we get the following interesting statement.
%
%\begin{cor}
%For $\pi:\FR{A}\rightarrow M$, if $\pi^\mathcal{U}(\FR{A})'\cap M^\mathcal{U}$ is a factor, then $W^*(\pi(\FR{A}))$ is a factor.
%\end{cor}

We now give an example of a face in $\HOM_w(C^*(\mathbb{F}_\infty),M)$ that does not come from an ideal.

\begin{exmpl}\label{nonidealface}
Let $M = R$, and let $\zeta: C^*(\mathbb{F}_\infty)\rightarrow R$ be an injective \linebreak $*$-homomorphism such that $\zeta(C^*(\mathbb{F}_\infty))$ is weakly dense in $R$ as provided by Proposition 3.1 of \cite{con}.  Then we consider \[\zeta^\mathcal{U}: C^*(\mathbb{F}_\infty)\rightarrow R^\mathcal{U}.\]  Consider $R\subset R^\mathcal{U}$ via the constant embedding. Since $\zeta$ has a dense image in $R$, we get that \[\zeta^\mathcal{U}(C^*(\mathbb{F}_\infty))'\cap R^\mathcal{U} = R'\cap R^\mathcal{U}.\]  It is well-known that $R'\cap R^\mathcal{U}$ is a factor.  Thus $\zeta^\mathcal{U}(C^*(\mathbb{F}_\infty))'\cap R^\mathcal{U}$ is a factor, and by Theorem \ref{comext}, we get that $[\zeta]$ is extreme.  So $\left\{[\zeta]\right\}$ is a face of $\HOM_w(C^*(\mathbb{F}_\infty),R)$ that does not factor through a quotient map.  
%Furthermore, if $\alpha$ is a non-approximately inner automorphism of $C^*(\mathbb{F}_\infty)$ (e.g. permutation of generators), then it follows from injectivity and density of image that $[\pi] \neq [\alpha \circ \pi]$.  So there are infinitely many extreme points of $\HOM_w(C^*(\mathbb{F}_\infty),R)$ that do not factor through a quotient map.
\end{exmpl}

%The argument in the above example concerning non-approximately inner automorphisms applies more generally.  Compare the following theorem with the discussion in section 6 of \cite{topdyn}.
%
%\begin{thm}
%Let $\FR{A}$ be such that Out$(\FR{A})$ is non-trivial and such that there exists an injective $*$-homomorphism $\pi: \FR{A}\rightarrow M$ with dense image.  Then Out$(\FR{A})$ acts nontrivially on $\HOM_w(\FR{A},M)$.  That is if $\alpha$ is non-approximately inner, then $[\alpha \circ \pi] \neq [\pi]$.
%\end{thm}

\subsection{Extreme Points: Amenability in First Argument}\label{amenable1st}

We first note the following theorem.

\begin{thm}\label{uatbetasurj}
If $T(\FR{A}) = \text{UAT}(\FR{A})$ then for any McDuff $M$, \[\HOM_w(\FR{A},M) \cong \HOM(\FR{A},M^\mathcal{U})\] via $\beta_{(\FR{A},M)}$.  In particular, if $\FR{A}$ is nuclear, then $\HOM_w(\FR{A},M) \cong \HOM(\FR{A},M^\mathcal{U})$.
\end{thm}

\begin{proof}
By Proposition \ref{betainj}, it suffices to show that $\beta_{(\FR{A},M)}$ is surjective.  We will let $\tau$ denote the unique tracial state on both $M$ and $M^\mathcal{U}$. Let $\pi: \FR{A}\rightarrow M^\mathcal{U}$ be given. Let $T \in T(\FR{A})$ be given by $T = \tau \circ \pi$.  Since $T \in \text{UAT}(\FR{A})$, there is a $\rho: \FR{A}\rightarrow M$ so that $T = \tau \circ \rho$.  Now, consider $\rho^\mathcal{U}: \FR{A} \rightarrow M^\mathcal{U}$.  Since the images of $\pi$ and $\rho$ are both hyperfinite, the argument from Theorem \ref{injR} gives that $\pi \sim \rho$.  By Theorem 3.1 of \cite{autoultra} we get that $\pi$ and $\rho$ are unitarily equivalent.  Thus $\beta_{(\FR{A},M)}([\rho]) = [\rho^\mathcal{U}] = [\pi]$.
\end{proof}

We would like to take this chance to show how the strategy of the proof of Theorem \ref{uatbetasurj} can be used to give a concise proof of an equivalent form of Theorem 4.10 from of Hadwin and Li's paper \cite{hadwin-li}.

\begin{thm}[Theorem 4.10, \cite{hadwin-li}]
Let $\left\{M_i\right\}$ be a collection of II$_1$-factors with traces $\tau_i$.  If $\FR{A}$ is either countably generated hyperfinite von Neumann algebra or a separable unital $C^*$-algebra such that all traces give hyperfinite GNS constructions ($T(\FR{A}) = \text{UAT}(\FR{A})$), then for any unital $*$-homomorphism \[\pi: \FR{A} \rightarrow \prod_\mathcal{U} M_i\] there exists unital $*$-homomorphisms $\pi_i: \FR{A} \rightarrow M_i$ such that for every $a \in \FR{A}$, \[\pi(a) = (\pi_i(a))_\mathcal{U}\] and \[\tau_\mathcal{U} \circ \pi = \tau_i \circ \pi_i\] for every $i$ where $\tau_\mathcal{U}$ denotes the trace on the ultraproduct.  

\vspace{.2cm}

\noindent In particular, homomorphisms from separable nuclear $C^*$-algebras into ultraproducts of II$_1$-factors lift to coordinate-wise $*$-homomorphisms.
\end{thm}

\begin{proof}
Let $\ds \pi: \FR{A} \rightarrow \prod_\mathcal{U} M_i$ be given.  Then as in the proof of Theorem \ref{uatbetasurj}, let $T = \tau_\mathcal{U} \circ \pi$.  Since $T$ induces a hyperfinite GNS construction, we can find unital $*$-homomorphisms $\rho_i: \FR{A} \rightarrow M_i$ so that $T = \tau_i \circ \rho_i$.  Note that an ultraproduct version of Theorem 3.1 of \cite{autoultra} holds with the same exact proof.  By uniqueness of GNS constructions we have that $\pi$ and $(\rho_i)_\mathcal{U}$ both have hyperfinite images, so as argued above in the proof of Theorem \ref{uatbetasurj}, using the ultraproduct version of Theorem 3.1 from \cite{autoultra}, we get that $\pi$ and $(\rho_i)_\mathcal{U}$ are unitarily equivalent.  Let $u$ be a unitary in $\ds \prod_\mathcal{U} M_i$ such that $\pi = u (\rho_i)_\mathcal{U} u^*$.  We may write $u = (u_i)_\mathcal{U}$ where each $u_i$ is a unitary in $M_i$.  Then we have $\pi = (u_i\rho_iu_i^*)_\mathcal{U}$.  So put $\pi_i = u_i \rho_i u_i^*$, and we are done.
\end{proof}

We now define a map $\tilde{\alpha}_{(\FR{A},M)}: \HOM(\FR{A},M^{\mathcal{U}}) \rightarrow T(\FR{A})$ analogous to $\alpha_{(\FR{A},M)}$ in the following way.

\begin{dfn}
Let \[\tilde{\alpha}_{(\FR{A},M)}: \HOM(\FR{A},M^{\mathcal{U}}) \rightarrow T(\FR{A})\] given by \[\tilde{\alpha}_{(\FR{A},M)}([\pi]) = \tau_{M^\mathcal{U}} \circ \pi\] where $\tau_{M^{\mathcal{U}}}$ is the unique normalized trace of $M^{\mathcal{U}}$.
\end{dfn}

The maps $\alpha_{(\FR{A},M)}$ and $\tilde{\alpha}_{(\FR{A},M)}$ relate naturally to one another via the map $\beta_{(\FR{A},M)}$.  In particular, we have that the diagram

\begin{equation}
\begin{tikzpicture}[baseline=(current bounding box.base)]\label{alphabeta}
	\matrix(m)[matrix of math nodes, row sep=2em,column sep = 2em,minimum width=2em]
		{ \HOM_w(\FR{A},M) & & T(\FR{A})\\
		& & \\
		\HOM(\FR{A},M^\mathcal{U}) & & \\};
	\path[-stealth]
		(m-1-1) edge node [above] {$\alpha_{(\FR{A},M)}$} (m-1-3)
		(m-1-1) edge node [left] {$\beta_{(\FR{A},M)}$} (m-3-1)
		(m-3-1) edge node [right] {\hspace{.1in}$\tilde{\alpha}_{(\FR{A},M)}$} (m-1-3);
\end{tikzpicture}
\end{equation}
commutes.

With the above notation in place, we have the following corollary to Theorem \ref{uatbetasurj}

\begin{cor}
If $T(\FR{A}) = \text{UAT}(\FR{A})$, then for any McDuff $M$, $\HOM(\FR{A},M^\mathcal{U}) \cong T(\FR{A})$ via $\tilde{\alpha}_{(\FR{A},M)}$. 
\end{cor}
%
%\noindent This gives us the following corollary. 
%
%\begin{cor}\label{betasurj}
%If $\FR{A}$ is nuclear, then $\HOM_w(\FR{A},M) \cong \HOM(\FR{A},M^\mathcal{U})$.
%\end{cor}

So we can immediately observe the following characterizations of extreme points in $T(\FR{A})$ and $\HOM_w(\FR{A},M)$ when $T(\FR{A}) = \text{UAT}(\FR{A})$.  

\begin{cor}\label{nucchar}
Let $\FR{A}$ be such that $T(\FR{A}) = \text{UAT}(\FR{A})$.

\begin{enumerate}

	\item $T \in T(\FR{A})$ is extreme if and only if $\pi_T(\FR{A})'\cap X^\mathcal{U}$ is a factor where $\pi_T:\FR{A}\rightarrow X^\mathcal{U}$ is a lift of $T$ through $X^\mathcal{U}$ for any separable II$_1$-factor $X$.

	\item The following are equivalent.

		\begin{enumerate}

			\item $[\pi] \in \HOM_w(\FR{A},M)$ is extreme;

			\item $\pi^\mathcal{U}(\FR{A})' \cap M^\mathcal{U}$ is a factor;  

			\item $W^*(\pi(\FR{A}))\subset M$ is a factor.
		
		\end{enumerate}

\end{enumerate}
\end{cor}
\noindent Note that the equivalence between (2b) and (2c) is a purely algebraic statement with no reference to $\HOM_w(\FR{A},M)$.

\subsection{Extreme Points: Amenability in Second Argument}\label{amenable2nd}

A satisfying characterization of extreme points is also available when we shift our amenability assumption to the second argument of $\HOM_w(\FR{A},M)$.  We state this in the following theorem.

\begin{thm}\label{hyperchar}
Let $\FR{A}$ be a (not necessarily nuclear) separable unital $C^*$-algebra.  Then given a $*$-homomorphism $\pi: \FR{A} \rightarrow R$, the following are equivalent.

\begin{enumerate}

	\item $[\pi] \in\HOM_w(\FR{A},R)$ is extreme;

	\item $W^*(\pi(\FR{A})) \subset R$ is a factor;

	\item $\pi^\mathcal{U}(\FR{A})' \cap R^\mathcal{U}$ is a factor.

\end{enumerate}
\end{thm}

\begin{proof}
((1) $\Rightarrow$ (2)): This is just Theorem \ref{extfact}.

\noindent ((2) $\Rightarrow$ (3)): We have that 
\[R \supset W^*(\pi(\FR{A})) \cong W^*(\pi^\mathcal{U}(\FR{A})) \subset R \subset  R^\mathcal{U}\]
and
\[W^*(\pi^\mathcal{U}(\FR{A}))'\cap R^\mathcal{U} = \pi^\mathcal{U}(\FR{A})'\cap R^\mathcal{U}.\]  
And since the factor $W^*(\pi(\FR{A})) \subset R$ must be separable, finite, and hyperfinite, Corollary \ref{finitefactorcommutant} or Theorem \ref{stuart} implies that $\pi^\mathcal{U}(\FR{A})'\cap R^\mathcal{U}$ must be a factor.

\noindent ((3) $\Rightarrow$ (1)): This is just Theorem \ref{comext}.
\end{proof}
\noindent Again, notice that the equivalence of (2) and (3) is a purely algebraic statement.

\begin{rmk}
In Example 6.4(2) of \cite{FHS}, the existence of a locally universal separable II$_1$-factor $S$ was established.  This $S$ has the property that any separable II$_1$-factor embeds into $S^\mathcal{U}$.  Tensoring $S$ with $R$ preserves this property, so we may assume that $S$ is McDuff.  Therefore, we may consider the convex structure $\HOM(N,S^\mathcal{U})$ for \emph{any} separable II$_1$-factor $N$ without any additional embeddability assumptions.
\end{rmk}

\section{Stabilization}\label{stabilization}

The \textquotedblleft McDuffness\textquotedblright of the codomain of $\pi: \FR{A} \rightarrow M$ allows us to coherently define the convex structure on $\HOM_w(\FR{A},M)$.  Only considering McDuff codomains seems to provide some restrictions on our theory and collection of examples. Unfortunately, without a tensor factor of $R$ in the target, it is unclear how to define a convex structure on $\HOM_w(\FR{A},N)$ for a non-McDuff $N$.

A natural way around this obstruction is to stabilize a given non-McDuff codomain.  That is, given a non-McDuff factor $N$ and a $*$-homomorphism \[\pi: \FR{A} \rightarrow N,\] we compose $\pi$ with the embedding \[\text{id}_N \otimes 1_R: N \rightarrow N \otimes R.\]  Using the notation from \S\S \ref{functoriality}, this composition induces the map \[(\text{id}_N \otimes 1_R)_*: \HOM_w(\FR{A},N) \rightarrow \HOM_w(\FR{A}, N\otimes R).\] That is, \[(\text{id}_N \otimes 1_R)_*([\pi]) = [(\text{id}_N \otimes 1_R) \circ \pi] = [\pi \otimes 1_R]\] where \[(\pi\otimes 1_R)(a) = \pi(a) \otimes 1_R.\] It turns out that $(\text{id}_N \otimes 1_R)_*$ is well-defined and injective: \[(\pi \otimes 1_R \sim \rho \otimes 1_R) \Leftrightarrow (\pi \sim \rho).\]  Well-defined is a clear observation.  Showing that $(\text{id}_N \otimes 1_R)_*$ is injective is not obvious at all.  The author would like to thank N. Ozawa for suggesting the proof of Theorem \ref{presineq}.  First we need the following fact established by Haagerup in Section 4 of \cite{newproof} concerning the notion of $\delta$-related $n$-tuples of unitaries.

\begin{dfn}[\cite{newproof}]
Let $N$ be a II$_1$-factor.  For $n \in \mathbb{N}$ and $\delta > 0$, two $n$-tuples $(u_1,\dots,u_n)$ and $(v_1,\dots,v_n)$ of unitaries in $N$ are \emph{$\delta$-related} if there is a sequence $\left\{a_j\right\} \subset N$ with
\begin{align}\label{summing}
\sum_j a_j^*a_j &=1 = \sum_j a_ja_j^*
\end{align}
such that for every $1 \leq k \leq n$,
\begin{align}\label{ineq}
\sum_j ||a_j u_k - v_ka_j||_2^2 &< \delta.
\end{align}
  We say that $\left\{a_j\right\}$ is a sequence that witnesses that $(u_1,\dots,u_n)$ and $(v_1,\dots,v_n)$ are $\delta$-related.
\end{dfn}

\begin{thm}[\cite{newproof}]\label{deltarelated}
Let $N$ be a II$_1$-factor.  For every $n \in\mathbb{N}$ and every $\varepsilon > 0$, there exists a $\delta(n,\varepsilon) > 0$ such that for any two $\delta(n,\varepsilon)$-related $n$-tuples of unitaries $(u_1,\dots,u_n)$ and $(v_1,\dots,v_n)$ in $N$, there exists a unitary $w \in N$ such that for every $1 \leq k \leq n$ \[||wu_k - v_kw||_2 < \varepsilon.\]
\end{thm}

\noindent Next we establish the following lemma.

\begin{lem}\label{tensordelta}
Let $N_1$ and $N_2$ be separable II$_1$-factors, and let $(u_1,\dots,u_n)$ and $(v_1,\dots, v_n)$ be two $n$-tuples of unitaries in $N_1$.  Fix $\delta >0$, and let $z \in N_1 \otimes N_2$ be a unitary of the form \[z = \sum_{j=1}^\infty a_j \otimes b_j\] where $\left\{b_j\right\} \subset N_2$ is an orthonormal basis in $L^2(N_2)$.  If $z$ is such that for every $1 \leq k \leq n$, \[||z(u_k\otimes 1_{N_2}) - (v_k\otimes 1_{N_2})z||_2^2 < \delta,\] then $(u_1,\dots,u_n)$ and $(v_1,\dots,v_n)$ are $\delta$-related.  Furthermore, $\left\{a_j\right\}$ is a sequence that witnesses that $(u_1,\dots,u_n)$ and $(v_1,\dots, v_n)$ are $\delta$-related.
\end{lem}

\begin{proof}

It suffices to show that $\ds \left\{ a_j\right\}$ is the sequence that witnesses that $(u_1,\dots,u_n)$ and $(v_1,\dots,v_n)$ are $\delta$-related.  First we show that the summing condition \eqref{summing} is satisfied.  This is a consequence of the fact that $\sum_{j} a_j \otimes b_j$ is a unitary.  Let $\mathbb{E}_1$ be the canonical normal conditional expectation \[\mathbb{E}_1: N_1 \otimes N_2 \rightarrow N_1 \otimes \mathbb{C}1_{N_2}\] onto the first tensor factor.  So on simple tensors, $\mathbb{E}_1(x \otimes y) = x \otimes \tau(y)$.  So we have

\begin{align*}
1 &= \mathbb{E}_1(z^*z)\\
& = \mathbb{E}_1(\sum_{j',j} a_{j'}^*a_j \otimes b_{j'}^*b_j)\\
& = \sum_{j',j} a_{j'}^*a_j \otimes \tau(b_{j'}^*b_j)\\
& = \sum_j a_{j}^*a_j \otimes 1.
\end{align*}

\noindent Thus $\sum_j a_j^*a_j = 1$ and by a symmetric argument, $\sum_j a_ja_j^* = 1$.

\noindent To check \eqref{ineq}, fix $1 \leq k \leq n$ and observe

\begin{align*}
\sum_{j} \Big|\Big|a_ju_k - v_ka_j\Big|\Big|_2^2&= \sum_{j} \tau\Big(\big(a_ju_k - v_ka_j\big)^*\big(a_ju_k - v_ka_j\big)\Big)\\
 &= \sum_{j',j} \tau\Big(\big((a_{j'}u_k - v_ka_{j'})^*(a_ju_k - v_ka_j)\big)\otimes b_{j'}^*b_j\Big)\\
 & = \tau\Big(\Big(\sum_{j'} (a_{j'}u_k - v_ka_{j'})\otimes b_{j'}\Big)^*\Big(\sum_{j}(a_ju_k - v_ka_j)\otimes b_j\Big)\Big)\\
 & = ||z(u_k \otimes 1_{N_2}) - (v_k\otimes 1_{N_2})z||_2^2\\ 
 & < \delta.\qedhere
\end{align*}

\end{proof}

\begin{thm}\label{presineq}
Let $N_1$ and $N_2$ be arbitrary separable II$_1$-factors. Given \linebreak $*$-homomorphisms $\pi,\rho: \FR{A}\rightarrow N_1$, consider $\pi \otimes 1_{N_2}, \rho \otimes 1_{N_2}: \FR{A} \rightarrow N_1 \otimes N_2$. If $\pi \otimes 1_{N_2} \sim \rho \otimes 1_{N_2}$ then $\pi \sim \rho$.
\end{thm}

\begin{proof}
Because a unital $C^*$-algebra is generated by its unitaries, it suffices to show that for any $\varepsilon > 0$ and any set of unitaries $u_1,\dots,u_n \in \FR{A}$, there exists a unitary $w\in N$ such that for every $1 \leq k \leq n$, \[||w\pi(u_k) - \rho(u_k)w||_2 < \varepsilon.\]  

Fix $\varepsilon > 0$ and unitaries $u_1,\dots, u_n \in \FR{A}$.  Let $\delta(n,\varepsilon)>0$ be such that if $(v_1,\dots,v_n)$ and $(v_1',\dots,v_n')$ are $\delta(n,\varepsilon)$-related $n$-tuples of unitaries in $N_1$, then there is a unitary $w \in N_1$ such that for every $1 \leq k \leq n$, \[||wv_k - v_k'w||_2 < \varepsilon\] as guaranteed by Theorem \ref{deltarelated}.  Thus we will be done if we show that $(\pi(u_1),\dots, \linebreak \pi(u_n))$ and $(\rho(u_1),\dots,\rho(u_n))$ are $\delta(n,\varepsilon)$-related.

Since $\pi \otimes 1_{N_2} \sim \rho \otimes 1_{N_2}$, we can find a unitary $z \in N_1 \otimes N_2$ such that for every $1 \leq k \leq n$, \[||z(\pi(u_k)\otimes 1_{N_2}) - (\rho(u_k)\otimes 1_{N_2})z||_2^2 < \delta(n, \varepsilon).\]  By standard approximation arguments we may assume that \[z = \sum_{j = 1}^\infty a_j \otimes b_j\] where $\left\{b_j\right\}\subset N_2$ is an orthonormal basis in $L^2(N_2)$ (guaranteed by Gram-Schmidt). Then by Lemma \ref{tensordelta} we have that $(\pi(u_1),\dots,\pi(u_n))$ and $(\rho(u_1),\dots, \rho(u_n))$ are $\delta(n,\varepsilon)$-related.

\end{proof}

\noindent  The following theorem appears as Corollary 3.3 in \cite{connes}; we mention it here because one can use the strategy from Theorem \ref{presineq} for an alternative proof.

\begin{thm}[\cite{connes}]
Let $N_1$ and $N_2$ be separable II$_1$ factors, and let $\theta_i \in \text{Aut}(N_i),  \linebreak i = 1,2$.  Then $\theta_1 \otimes \theta_2 \in \text{Aut}(N_1\otimes N_2)$ is approximately inner if and only if $\theta_1$ and $\theta_2$ are both approximately inner.
\end{thm}

%\begin{proof}
%We will show the contrapositive.  That is, if $\beta \otimes \text{id}_{N_2}$ is approximately inner, then $\beta$ is approximately inner.  Since $N_1$ is a II$_1$-factor, it is generated by its unitaries, so it will suffice to show that for any $\varepsilon >0$ and any set of unitaries $u_1,\dots, u_k$, there exists a unitary $w \in N_1$ such that for every $1\leq k \leq n$, \[||w\beta(u_k) - u_kw||_2 < \varepsilon.\]
%
%Fix $\varepsilon > 0$ and unitaries $u_1,\dots, u_n \in N_1$.  Let $\delta(n,\varepsilon)>0$ be such that if $(v_1,\dots,v_n)$ and $(v_1',\dots,v_n')$ are $\delta(n,\varepsilon)$-related $n$-tuples of unitaries in $N_1$, then there is a unitary $w \in N_1$ such that for every $1 \leq k \leq n$, \[||wv_k - v_k'w||_2 < \varepsilon\] as guaranteed by Theorem \ref{deltarelated}.  Thus we will be done if we show that $(\beta(u_1),\dots,\beta(u_n))$ and $(u_1,\dots,u_n)$ are $\delta(n,\varepsilon)$-related.
%
%Since $\beta \otimes \text{id}_{N_2} \sim \text{id}_{N_1\otimes N_2}$, we can find a unitary $z \in N_1 \otimes N_2$ such that for every $1 \leq k \leq n$, \[||z(\beta\otimes \text{id}_{N_2})(u_k\otimes 1_{N_2}) - (u_k\otimes 1_{N_2})z||_2^2 < \delta(n, \varepsilon).\]  By standard approximation arguments we may assume that \[z = \sum_{j = 1}^\infty a_j \otimes b_j\] where $\left\{b_j\right\}\subset N_2$ is an orthonormal basis in $L^2(N_2)$ (guaranteed by Gram-Schmidt). Then by Lemma \ref{tensordelta} we have that $(\beta(u_1),\dots,\beta(u_n))$ and $(u_1,\dots, u_n)$ are $\delta(n,\varepsilon)$-related. 
%
%\end{proof}

\begin{exmpl}
It is well-known that $L(\mathbb{F}_2)$ has non-approximately inner automorphisms (e.g. exchanging generators).  Let $\beta$ be such a non-approximately inner automorphism of $L(\mathbb{F}_2)$.  Considering $C^*_r(\mathbb{F}_2) \subset L(\mathbb{F}_2)$, we have \[\pi := \text{id}_{C^*_r(\mathbb{F}_2)} \nsim \beta \circ \text{id}_{C^*_r(\mathbb{F}_2)} =: \rho.\]  Then by Theorem \ref{presineq} we have that $[\pi \otimes 1_R] \neq [\rho\otimes 1_R]$.  Since $\HOM_w(C^*_r(\mathbb{F}_2),L(\mathbb{F}_2) \otimes R)$ is convex, there is at least an interval's worth, \[\left\{ t[\pi \otimes 1_R] + (1-t)[\rho \otimes 1_R]: t\in [0,1]\right\},\] of inequivalent $*$-homomorphisms of $C^*_r(\mathbb{F}_2)$ into $L(\mathbb{F}_2) \otimes R$.
\end{exmpl}

For the next theorem, we will use the following standard fact about complete metric spaces.

\begin{prop}\label{metricprop}
Let $(X,d)$ and $(Y,d')$ be complete metric spaces.  If $\varphi: X \rightarrow Y$ satisfies the following conditions
\begin{enumerate}
\item $\varphi$ is continuous, 
\item $\varphi$ is injective,
\item $\left\{\varphi(x_n)\right\}$ is Cauchy in $d' \Rightarrow \left\{x_n\right\}$ is Cauchy in $d$;
\end{enumerate}
then $\varphi$ is a homeomorphism onto its image and $\varphi(X)$ is closed in $Y$.
\end{prop}

\begin{thm}\label{nonmcd}
For any two separable II$_1$-factors $N_1$ and $N_2$, the map \[(\text{id}_{N_1} \otimes 1_{N_2})_*: \HOM_w(\FR{A},N_1) \rightarrow \HOM_w(\FR{A}, N_1\otimes {N_2})\] is a homeomorphism onto its image.  Furthermore, $(\text{id}_{N_1} \otimes 1_{N_2})_*(\HOM_w(\FR{A},N_1))$ is closed in $\HOM_w(\FR{A},N_1\otimes {N_2})$. In particular, we may consider $\HOM_w(\FR{A},N_1)$ as a closed subset of $\HOM_w(\FR{A},N_1\otimes {N_2})$.
\end{thm}

\begin{proof}
Let $\left\{u_k\right\}$ be a sequence of unitaries that generate $\FR{A}$.  As in Remark \ref{metricgenerators}, we can use these unitaries to define metrics $d_{N_1}$ and $d_{N_1\otimes {N_2}}$ on $\HOM_w(\FR{A},N_1)$ and $\HOM_w(\FR{A},N_1\otimes {N_2})$ respectively.  That is,
\[d_{N_1}([\pi],[\rho]) = \inf_{v \in \mathcal{U}(N_1)} \Big(\sum_{k=1}^{\infty} \frac{1}{2^{2k}} || v\pi(u_k) -  \rho(u_k)v||^2_2\Big)^{\frac{1}{2}}\]
and
\[d_{N_1\otimes {N_2}}([\pi],[\rho]) = \inf_{z \in \mathcal{U}(N_1\otimes {N_2})} \Big(\sum_{k=1}^{\infty} \frac{1}{2^{2k}} || z\pi(u_k) -  \rho(u_k)z||^2_2\Big)^{\frac{1}{2}}.\]

Theorem \ref{presineq} shows that $(\text{id}_{N_1} \otimes 1_{N_2})_*$ is injective, and $(\text{id}_{N_1} \otimes 1_{N_2})_*$ is continuous by Proposition \ref{psitilde}.  Now, by Proposition \ref{metricprop}, it suffices to show that if $\left\{[\pi_n\otimes 1_{N_2}]\right\}$ is Cauchy in $d_{N_1\otimes {N_2}}$ then $\left\{[\pi_n]\right\}$ is Cauchy in $d_{N_1}$.  Fix $\varepsilon > 0$.  Let $J \in \mathbb{N}$ be such that $\ds \sum_{k = J+1}^\infty \frac{4}{2^{2k}} < \frac{\varepsilon^2}{2}$. By Theorem \ref{deltarelated} there is a $\ds \delta\Big(J, \frac{\varepsilon}{\sqrt{2J}}\Big)$ such that for any pair of $\ds \delta\Big(J, \frac{\varepsilon}{\sqrt{2J}}\Big)$-related $n$-tuples of unitaries $(w_1,\dots,w_J)$ and $(w'_1,\dots,w'_J)$ in $N_1$ there is a unitary $v \in N_1$ such that for every $1 \leq k \leq J$,
\[||vw_k - w'_kv||_2 < \frac{\varepsilon}{\sqrt{2J}}\] 
or 
\[||vw_k - w'_kv||_2^2 < \frac{\varepsilon^2}{2J}.\]  Now let $K \in \mathbb{N}$ be such that for $n,m \geq K$ we have \[d_{N_1\otimes {N_2}}([\pi_n \otimes 1_{N_2}],[\pi_m \otimes 1_{N_2}])^2 < \frac{\delta\Big(J,\frac{\varepsilon}{\sqrt{2J}}\Big)}{2^{2J}}.\]  We will show that for any $n,m \geq K, d_{N_1}([\pi_n],[\pi_m]) < \varepsilon$. Fix $n,m \geq K$. From the definition of $d_{N_1\otimes {N_2}}$ there is a unitary $z \in N_1 \otimes {N_2}$ of the form $\ds z = \sum_j a_j\otimes b_j$ with $\left\{b_j \right\}$ an orthonormal basis in $L^2({N_2})$ such that \[\sum_{k=1}^{\infty} \frac{1}{2^{2k}} || z((\pi_n \otimes 1_{N_2})(u_k)) -  ((\pi_m\otimes 1_{N_2})(u_k))z||^2_2 < \frac{\delta\Big(J,\frac{\varepsilon}{\sqrt{2J}}\Big)}{2^{2J}}.\] So, for $1 \leq k' \leq J$ we have 
\[\frac{1}{2^{2k'}}||z((\pi_n \otimes 1_{N_2})(u_{k'})) -  ((\pi_m\otimes 1_{N_2})(u_{k'}))z||^2_2\]
\begin{align*}
 & \leq  \sum_{k=1}^{\infty} \frac{1}{2^{2k}} || z((\pi_n \otimes 1_{N_2})(u_k)) -  ((\pi_m\otimes 1_{N_2})(u_k))z||^2_2\\
& < \frac{\delta\Big(J,\frac{\varepsilon}{\sqrt{2J}}\Big)}{2^{2J}}.
\end{align*}
Therefore for every $1 \leq k \leq J$, \[||z((\pi_n \otimes 1_{N_2})(u_k)) -  ((\pi_m\otimes 1_{N_2})(u_k))z||^2_2 < 2^{2k} \frac{\delta\Big(J,\frac{\varepsilon}{\sqrt{2J}}\Big)}{2^{2J}} \leq \delta\Big(J,\frac{\varepsilon}{\sqrt{2J}}\Big).\]
Thus by Lemma \ref{tensordelta} we have that $(\pi_n(u_1),\dots,\pi_n(u_J))$ and $(\pi_m(u_1),\dots,\pi_m(u_J))$ are $\ds \delta\Big(J,\frac{\varepsilon}{\sqrt{2J}}\Big)$-related.  By Theorem \ref{deltarelated}, there is a unitary $v \in N_1$ such that for every $1\leq k \leq J$, 
\[||v\pi_n(u_k) - \pi_m(u_k)v||_2^2 < \frac{\varepsilon^2}{2J}.\]  So to complete the proof we observe that,
\begin{align*}
d_{N_1}([\pi_n],[\pi_m])^2 &\leq \sum_{k=1}^\infty \frac{1}{2^{2k}}||v\pi_n(u_k) - \pi_m(u_k)v||_2^2\\
&= \sum_{k=1}^J \frac{1}{2^{2k}}||v\pi_n(u_k) - \pi_m(u_k)v||_2^2\\
 &\quad+ \sum_{k=J+1}^\infty \frac{1}{2^{2k}}||v\pi_n(u_k) - \pi_m(u_k)v||_2^2\\
& < J\cdot \frac{\varepsilon^2}{2J} + \frac{\varepsilon^2}{2}\\
& = \varepsilon^2. \qedhere
\end{align*}
\end{proof}

\noindent Therefore, if $N$ is an arbitrary separable II$_1$-factor, we may consider $\HOM_w(\FR{A},N)$ to be embedded as a closed subset of the convex set $\HOM_w(\FR{A},N\otimes R)$.  

\begin{exmpl}\label{fullhull}
When $\FR{A}$ is such that UAT$(\FR{A}) = T(\FR{A})$, $\HOM_w(\FR{A},N) \cong T(\FR{A})$; so \[(\text{id}_N \otimes 1_R)_*(\HOM_w(\FR{A},N)) = \HOM_w(\FR{A},N\otimes R).\]
\end{exmpl}

\begin{exmpl}
If $N$ is a non-hyperfinite solid II$_1$-factor, for example $L(\mathbb{F}_2)$ or $L(\mathbb{Z}^2 \rtimes \text{SL}(2,\mathbb{Z}))$, then all of its subfactors are also solid, see \cite{ozawasolid} and \cite{solidexample}.  Non-hyperfinite solid factors are prime.  So we know that $N \otimes R$ does not embed into $N$.  Thus, if we let $\FR{A}$ be a separable dense $C^*$-subalgebra of $N \otimes R$ with a unique trace, then $\HOM_w(\FR{A},N)$ is empty, but $\HOM_w(\FR{A}, N\otimes R)$ is nonempty. Such a dense monotracial $C^*$-subalgebra exists by applying the argument found in Lemma \ref{sepsubfac} except replace all $W^*$'s with $C^*$'s and all instances of \tql weak\tqr with \tql norm.\tqr \, Therefore, in contrast to the situation of Example \ref{fullhull}, we have that $(\text{id}_M \otimes 1_R)_*(\HOM_w(\FR{A},N))$ is empty.
\end{exmpl}

\noindent It would be interesting to know more about the way $\HOM_w(\FR{A},N)$ sits inside \linebreak $\HOM_w(\FR{A}, N \otimes M)$ for $M$ McDuff and $\FR{A}$ such that UAT$(\FR{A}) \neq T(\FR{A})$.
%It would be interesting to know if, in general, the closed convex hull of $\HOM_w(\FR{A},N)$ inside $\HOM_w(\FR{A},N\otimes R)$ is all of $\HOM_w(\FR{A},N\otimes R)$.    Alternatively, we can consider embedding $N$ into an infinite tensor product of II$_1$-factors such as $\ds \otimes_{i=1}^\infty N_i$ where $N_i = N$ for every $i$.  Note that infinite tensor products always have non-hypercentral central sequences and are therefore automatically McDuff--c.f. \cite{mcduff}.  This situation could lead to some interesting questions concerning the dynamics of the convex space $\ds \HOM_w(\FR{A}, \otimes_{i=1}^\infty N_i)$.

\section{Open Questions}\label{oq}

We conclude this paper with a brief discussion of some open questions concerning $\HOM_w(\FR{A},M)$.
\vspace{.1in}

We do not yet have a characterization of extreme points in $\HOM_w(\FR{A},M)$ in general.  So our first question is the following.

\begin{?}\label{satchar}
What is a necessary and sufficient condition for $[\pi] \in \HOM_w(\FR{A},M)$ to be extreme?
\end{?}

\noindent Recall that there is a natural embedding $\HOM_w(\FR{A},M) \subseteq \HOM(\FR{A},M^\mathcal{U})$ via constant sequences (see Proposition \ref{betainj}).  It would be very useful to answer the following question regarding this inclusion.

\begin{?}\label{embedface}
Is the image of $\HOM_w(\FR{A},M)$ under this natural embedding a face of $\HOM(\FR{A},M^\mathcal{U})$?
\end{?}

\noindent By Theorem \ref{ultraextreme}, a positive answer to Question \ref{embedface} would provide an answer to Question \ref{satchar}.  We also have the following open problem.

\begin{?}\label{existence}
Does every $\HOM_w(\FR{A},M)$ have extreme points?
\end{?}

\noindent In the context of $\HOM(\FR{A},M^\mathcal{U})$ or $\HOM(N,R^\mathcal{U})$, the existence of extreme points is also open (and well-known in the latter situation)--see Remark \ref{existence1}. Lastly, we record a natural question coming from the results of \S \ref{stabilization}.

\begin{?}
Let $N$ be an arbitrary II$_1$-factor, and let $M$ be a McDuff II$_1$-factor (e.g. $M = R$ or $\ds M = \otimes_{i=1}^\infty N_i$ where $N_i = N$ for all $i$).  Can the set $(\text{id}_N \otimes 1_M)_*(\HOM_w(\FR{A},N))$ ever fail to be convex? What is the relationship between this set, its convex hull, and $\HOM_w(\FR{A}, N \otimes M)$?
\end{?}

\bibliographystyle{plain}
\bibliography{convexhom}{}

\end{document}